\documentclass[a4paper, 11pt, reqno]{amsart}
\usepackage{a4wide}
\usepackage[usenames,dvipsnames]{xcolor}
\usepackage{amssymb, amsmath, latexsym, graphics, graphicx}
\usepackage{xspace}
\usepackage{tikz, ytableau}
\usepackage[normalem]{ulem}

\newtheorem{theorem}{Theorem}[section]
\newtheorem{lemma}[theorem]{Lemma}

\newtheorem{proposition}[theorem]{Proposition}
\theoremstyle{definition}
\newtheorem{example}[theorem]{Example}
\newtheorem{remark}[theorem]{Remark}
\newtheorem{definition}[theorem]{Definition}
\newtheorem{question}[theorem]{Question}

\allowdisplaybreaks
\numberwithin{equation}{section}

\makeatletter
\renewcommand{\author@andify}{%
  \nxandlist {\unskip ,\penalty-1 \space\ignorespaces}%
    {\unskip {} \@@and~}%
    {\unskip \penalty-2 \space \@@and~}%
}
\renewcommand{\andify}{\nxandlist{\unskip, }{\unskip{} \@@and~}{\unskip \space \@@and~}}
\newcommand{\raiseme}[2][.1em]{\def\r@iseval{#1}\mathpalette\r@iseme{#2}}
\def\r@iseme#1#2{\raise\r@iseval\hbox{$\m@th#1#2$}}
\newcommand{\lowerme}[2][.15em]{\def\lowerv@l{#1}\mathpalette\l@werme{#2}}
\def\l@werme#1#2{\lower\lowerv@l\hbox{$\m@th#1#2$}}
\makeatother

\def\rk{{\rm rk}\,}
\def\N{{\mathbb N}}
\newcommand{\T}{\mathcal{T}}					% For the full transformation monoid
\newcommand{\I}{\langle \im I\rangle}					% For the 
\renewcommand{\S}{\mathcal{S}}		%for the symmetric group
\newcommand{\op}{{\mkern 1mu\raiseme{*}}}	% For better alignment of * on \T^*
\newcommand{\fullTA}[1][A]{\T_{\mkern -2mu\lowerme{#1}}} % For better alignment of \T_A

\DeclareMathOperator{\en}{End}
\DeclareMathOperator{\im}{Im}
\let\ker\relax \DeclareMathOperator{\ker}{Ker} 

\newcommand{\EndA}[1][A]{\mathrm{End}(#1)}

\newcommand{\Rho}{\mathrm{P}}				% Capital Rho
\newcommand{\Rhot}{\widetilde{\Rho}}			% Rho tilde
\newcommand{\RhoS}[1][S]{\Rho_{\mkern -4mu\lowerme{#1}}} % Rho_S
\newcommand{\Lambdat}{\widetilde{\Lambda}}	% Lambda tilde
\newcommand{\LambdaS}[1][S]{\Lambda_{\lowerme{#1}}} % \Lambda_S
\newcommand{\OmegaS}[1][S]{\Omega_{\lowerme{#1}}} % \Omega_S
\newcommand{\piR}{\pi_{\lowerme{\mkern -1mu\Rho}}} % map \pi_\Rho
\newcommand{\piL}{\pi_{\lowerme{\mkern -3mu\Lambda}}} % map \pi_\Lambda
\newcommand{\chiR}{\chi_{\lowerme{\Rho}}}	% map \chi_\Rho
\newcommand{\chiL}{\chi_{\mkern -1mu\lowerme{\Lambda}}} % map \chi_\Lambda
\newcommand{\chiO}{\chi_{\mkern 1mu\lowerme{\Omega}}} % map \chi_\Omega

\newcommand{\set}[1]{\ensuremath{\left\lbrace #1 \right\rbrace}\xspace}
\newcommand{\clotX}[1][X]{\langle #1 \rangle}
\newcommand{\Card}[1]{\lvert #1 \rvert}
\newcommand{\isom}{\cong}

\begin{document}
\title[Translational hulls of semigroups of  endomorphisms]{Translational hulls of semigroups of  endomorphisms of an  algebra}
\author[V. Gould]{Victoria Gould}
\address{University of York
}
\email{victoria.gould@york.ac.uk}
\author[A. Grau]{Ambroise Grau}
\address{University of York}
\email{ambroise.grau@alumni.york.ac.uk}
\thanks{}
\author[M. Johnson]{Marianne Johnson}
\address{%
	University of Manchester\\
}
\email{Marianne.Johnson@manchester.ac.uk}

\author[M. Kambites]{Mark Kambites}
\address{%
	University of Manchester\\}
\email{Mark.Kambites@manchester.ac.uk}

\date{\today}
\begin{abstract}
We consider the translational hull $\Omega(I)$ of an arbitrary subsemigroup $I$ of an endomorphism monoid $\en(A)$ where $A$ is a universal algebra. We give  conditions for every bi-translation of $I$ to be realised by transformations,  or by endomorphisms, of $A$. We demonstrate that certain of these conditions are also sufficient to provide natural isomorphisms between the translational hull of $I$ and the  idealiser of $I$ within  $\en(A)$,  which in the case where $I$ is an ideal is simply $\en(A)$. We describe the connection between these conditions and work of Petrich and Gluskin  in the context of densely embedded ideals. Where the conditions fail, we develop a methodology to extract information concerning $\Omega(I)$ from the translational hull $\Omega(I/{\approx})$ of a quotient $I/{\approx}$ of $I$. We illustrate these concepts in detail in the cases where $A$ is: a free algebra; an independence algebra; a finite symmetric group.
\end{abstract}


\maketitle
\section{Introduction}

The  {\em translational hull} $\Omega(I)$ of  a semigroup $I$ is major tool in the theory of ideal extensions in semigroup theory. If $I$ is an ideal of a larger semigroup $U$, then the left and right actions of the elements of $U$ on those of $I$, given by multiplication in the semigroup $U$, are linked. This notion of linked pairs, or bi-translations, which we explain below, may be applied to any pair of actions, or {\em translations} of $I$, leading to the translational hull $\Omega(I)$. 
There is an extensive literature on the nature of $\Omega(I)$, in particular relating it to the theory of  dense extensions and, more generally,  to the theory of  ideal extensions. Dense extensions originate from the work of Lyapin \cite{Lyapin:1953}, whose aim was to abstractly characterise natural semigroups of transformations. This led  Gluskin \cite{Gluskin:1960,Gluskin:1961} to introduce the notion of translational hull. The early work was surveyed by Petrich \cite{Petrich:1970}, who gives a comprehensive list of references up to that time.
Subsequently the thrust has been to investigate translational hulls of semigroups in certain classes, or with additional structure, see, for example, \cite{Guo:2003, Hildebrandt:1976,Schein:1973}.

Our approach is somewhat different. We focus  on the case where $A$ is a (universal) algebra and $I$ a subsemigroup of the endomorphism monoid $\en(A)$ of $A$ and hence of  the full transformation monoid $\fullTA$ of all maps from $A$ to $A$. The context allows us to use the explicit nature of the elements of $I$, $\en(A)$ 
and $\fullTA$ as being maps on the  underlying set $A$, providing extra leverage.
Our study is motivated by the following question:  \emph{when is
$\Omega(I)$  isomorphic to $\en(A)$?} Where the answer is positive, this tells us that $\en(A)$ is determined by its ideal $I$, together with the natural left and right actions of $\en(A)$ on $I$. 

We sharpen this question as follows.
If $\rho$ is a right translation of $I$ satisfying $\alpha\rho=\alpha f$ for all $\alpha\in I$, for some $f$ in an oversemigroup of $I$, then we say $\rho$ is {\em realised} by $f$, and write $\rho=\rho_f$ to denote this. 
Likewise, if $\lambda$ is a left translation of $I$ satisfying $\lambda(\alpha)=f\alpha $ for all $\alpha\in I$ and some fixed $f$ in an oversemigroup of $I$, we say $\lambda$ is {\em realised} by $f$, and write $\lambda=\lambda_f$. If for a pair of left and right translations $(\lambda, \rho)$ there exists $f$ such that $\lambda =\lambda_f$ and $\rho=\rho_f$, then $(\lambda, \rho)=(\lambda_f,\rho_f)$ is a linked pair, and we say $f$ {\em realises} the bi-translation $(\lambda,\rho)$. As above, our main focus  is on  the case where  $I$ a subsemigroup of the endomorphism monoid $\en(A)$ of an algebra $A$. In Theorem~\ref{thm:maps} we give necessary and sufficient conditions for all bi-translations to be realised by transformations of $A$. That is, we determine when $\Omega(I) = \{(\lambda_f, \rho_f): f \in  T(A, I)\}$ where $T(A,I)$ is the idealiser of $I$ in $\fullTA$.  In many natural examples the union of the images of endomorphisms lying in $I$ generate the algebra $A$: in this case we say that $A$ is $I$-representable. Our main result is Theorem~\ref{thm:main}, which shows that $I$-representability, together with an existing condition called $I$-separability,  guarantee that $\Omega(I)\cong S(A,I)$, where $S(A,I)$ is the idealiser of $I$ in $\en(A)$. Often, we will start with $I$ being an ideal of $\en(A)$, so that 
$\en(A)=S(A,I)$.  In the case where $I$ is an ideal of $\en(A)$ and $I$-representability and $I$-separability hold we therefore have the desired outcome that $\Omega(I)\cong \en(A)$.  In the case where $I$-representability holds but $I$-separability does not, we develop a methodology to extract information concerning $\Omega(I)$ from the translational hull $\Omega(I/{\approx})$ of a quotient $I/{\approx}$ of $I$. 

The early catalyst for this study was several questions posed by Stuart Margolis concerning the nature of the translational hulls of ideals of some specific endomorphism monoids. We are able to answer those questions, and more, using the techniques we have developed. In particular, in Subsection~\ref{sub:free} we consider free algebras, in Subsection~\ref{sub:ind} the special case of independence algebras, and in Subsection~\ref{sec:Sn} the (very different case of) a finite symmetric group. 

The paper is structured as follows. In Section~\ref{sec:induced} we give the necessary background and definitions to allow us to frame our questions.  We introduce $I$-representability and $I$-separability, and explain the connections between our conditions, their variations,  and classical notions such   as left and right reductivity.  In Section~\ref{sec:maps} we present conditions for $\Omega(I) = \{(\lambda_f, \rho_f): f \in F\}$ where $F = T(A, I)$ or $S(A, I)$, and consider conditions under which $\Omega(I)$ is in fact isomorphic to $F$. In particular, we  show that the conditions $I$-representability and $I$-separability together are sufficient to yield $\Omega(I) \cong S(A, I)$.  In Section~\ref{sec:nosep} we investigate the situation where the $I$-representability condition holds but the $I$-separability condition does not. In this case we find  congruences $\sim$ on $A$ and 
$\approx$ on $I$ such that $I/{\approx}$ is a semigroup of endomorphisms of $A/{\sim}$  and $A/{\sim}$ has both  $I/{\approx}\,$-representability and 
$I/{\approx}\,$-separability. Further,
there is a natural morphism from $\Omega(I)$ to $\Omega(I/{\approx})$. Finally, in Section~\ref{sec:free algebras}, we look at several examples in detail, and apply the results of the preceding sections. Throughout we raise a number of questions and indicate avenues for future progress. We hope that this article will be a catalyst for new progress in the area of translational hulls.

We aim to make this article as self-contained as possible. By an {\em algebra} we mean a set $A$ together with a collection of finitary operations.
We denote the subalgebra whose elements are the images of the nullary operations of $A$ by $C$; note that $C=\langle\emptyset\rangle$ and $C=\emptyset$ if and only if there are no (basic) nullary operations.
An {\em endomorphism} of  $A$  is a transformation of  $A$ that preserves those operations; and a {\em congruence } on $A$ is an equivalence that is compatible with those operations. We make the convention that  all ideals and subsemigroups  are non-empty. Functions are applied on the right of their arguments, and composed from left to right, except where otherwise stated. We denote by $\mathbb{N}_0$ ($\mathbb{N}$) the set of (strictly) positive integers.
We introduce all remaining notions required in Section 2, making use of standard terminology.  For further background and details we refer to \cite{CP1}, \cite{CP2} and \cite{Howie:1995} for general semigroup theoretic definitions, to \cite{Burris:1981} for an introduction to universal algebra, and to \cite{Petrich:1970} for a detailed survey of results on translational hulls.

\section{Preliminaries}\label{sec:induced}

\subsection{Translations} \label{sub:definitions} We begin by formally defining the translational hull and setting up the notation for certain semigroups associated with the translational hull of a subsemigroup $I$ of a semigroup $U$, and introducing morphisms which will allow us to state our results in full.

For any semigroup $I$, a {\em right translation} of $I$ is a map $\rho: I \rightarrow I$ (applied on the right of its argument) satisfying $(\alpha\beta)\rho = \alpha(\beta \rho)$ for all $\alpha, \beta \in I$.  Likewise, a {\em left translation} of $I$ is a  map  $\lambda: I \rightarrow I$ (applied on the left of its argument) satisfying $\lambda(\alpha\beta) = \lambda(\alpha)\beta$ for all $\alpha, \beta \in I$.
A pair $(\lambda, \rho)$ is said to be a {\em bi-translation} of $I$ if $\lambda $ is a left translation of $I$, $\rho$ is a right translation of $I$ and the pair $(\lambda, \rho)$ is \textit{linked} via $\alpha\lambda(\beta) = (\alpha\rho)\beta$ for all $\alpha, \beta \in I$.

We write $\Rho(I)$ to denote the set of all right translations of $I$, $\Lambda(I)$ to denote the set of all left translations of $I$ and $\Omega(I)$ to denote the set of all bi-translations of $I$. 
It is immediate that $\Rho(I)$, $\Lambda(I)$ and $\Omega(I)$ are submonoids of $\T_I$, $\T_I^\op$ and $\T_I^\op \times \T_I$, respectively, where $\T_I$ is the full transformation semigroup (of right maps) of $I$ and $\T_I^\op$ its dual (of left maps). 
The semigroup $\Omega(I)$ is the {\em translational hull} of $I$. 

Recall that if $I$ is a subsemigroup of a semigroup $U$, then the {\em left idealiser} of $I$ in $U$ is $\{ s\in U: sI\subseteq I\}$, which is easily seen to be the largest subsemigroup of $U$ in which $I$ is a left ideal. The {\em right idealiser} of $I$ in $U$ is defined dually, and the 
{\em idealiser} of $I$ in $U$ is the intersection of left and right idealisers; it is the largest subsemigroup of $U$ in which $I$ is a two-sided ideal.

For $f\in U$ let $\rho_f:I\rightarrow U$ (applied on the right of its argument) and  $\lambda_f:I\rightarrow U$ (applied on the left of its argument) be the maps induced by right and left multiplication with $f$ within $U$, that is, for all $\alpha \in I$ we define $\alpha\rho_f=\alpha f$ and $\lambda_f \alpha =f\alpha$.
Then $\rho_f\in \Rho(I)$ [respectively, $\lambda_f\in \Lambda(I), \, (\lambda_f,\rho_f)\in \Omega(I)$] if and only if $f$ is in the right [respectively, left, (two-sided)] idealiser of $I$
in $U$ and we say that $f$ {\em induces} or, particularly where $U$ is a semigroup of mappings, {\em realises} $\rho_f$ [respectively, $\lambda_f,\, (\lambda_f,\rho_f)$].   
The set of all   right [respectively, left, bi-] translations induced by elements of $V$ where $V$ is any subsemigroup of the idealiser of $I$ in $U$ with $I \subseteq V$ is denoted by $\RhoS[V](I)$ [respectively, $\LambdaS[V](I)$, $\OmegaS[V](I)$].
We  denote the  projection map of $\Omega(I)$ onto the first [respectively, second] coordinate by $\piL$ [respectively,  $\piR$] and denote the image of $\piR$ [respectively,  $\piL$] by 
$\Rhot(I)$ [respectively,  $\Lambdat(I)$].  We have the inclusions
\[\RhoS[I](I)\subseteq \RhoS[V](I)\subseteq \Rhot(I) \subseteq \Rho(I),\qquad \LambdaS[I](I)\subseteq \LambdaS[V](I)\subseteq \Lambdat(I) \subseteq \Lambda(I)  \]
and  \[\OmegaS[I](I)\subseteq \OmegaS[V](I)\subseteq \Omega(I).\]
Note that all the sets mentioned above are semigroups, and $\piR$ and $\piL$ are morphisms.

In the case where  $I$ is an ideal of a semigroup $V$ we also define
 	\[\chiR^V: V\rightarrow \RhoS[V](I),\quad \chiL^V: V\rightarrow \LambdaS[V](I) \quad
 	\mbox{and} \quad \chiO^V:V\rightarrow \OmegaS[V](I)\]
 	by  
 	\[f\chiR^V = \rho_f, \quad f\chiL^V = \lambda_f \quad\mbox{and}\quad f\chiO^V = (\lambda_f,\rho_f). \]

We refer to  $\piR, \piL, \chiR^V, \chiL^V$ and $\chiO^V$ as {\em natural morphisms}. If, for example, $\piR$ is an isomorphism, we say it is a {\em natural isomorphism} and $\Omega(I)$ is {\em naturally isomorphic  to} $\Rhot(I)$.

\subsection{Idealisers for subsemigroups of $\en(A)$}\label{sub:idealisers}

We now focus on the case where $I$ is a semigroup of endomorphisms of an algebra $A$.

\textbf{Notation:} We have observed that $\rho_f$ [respectively, $\lambda_f,\, (\lambda_f,\rho_f)$] is a right [respectively, left, \,bi-] translation if and only if
$f$ is in the right [respectively, left, (two-sided)] idealiser of $I$ in $\fullTA$.  Later it will be convenient to write $T(A,I)$ to denote the (two-sided) idealiser of $I$ in $\fullTA$ and $S(A,I)$ to denote the (two-sided) idealiser of $I$ in $\en(A)$; where these appear as subscripts or superscripts we abbreviate as $T:=T(A,I)$ and $S:=S(A,I)$.

We now record several observations about right/left/bi-translations induced by composition with a transformation of $A$.

\begin{lemma}\label{lem: right_trans}
Let $A$ be an algebra, $I$ a subsemigroup of $\en(A)$, and let $f, g \in \fullTA$. 
\begin{enumerate}
\item We have that $f$ and $g$ induce the same map by right multiplication on $I$ ($\rho_f = \rho_g$) if and only if $f|_{\im\alpha} = g|_{\im\alpha}$ for all $\alpha \in I$.

\item For any $\alpha\in \en(A)$, we have that $\alpha f\in\en(A)$ if and only if 
$f|_{\im\alpha}$ is a morphism.
\item If $I$ is a right ideal of $\en(A)$ with $I^2=I$, then
$\rho_f$ is a right translation of $I$ if and only if $f|_{\im\alpha}$ is a morphism for all $\alpha \in I$.

\item We have that $f$ and $g$ induce the same map by left multiplication on $I$ ($\lambda_f=\lambda_g$) if and only if $(af,ag)\in \ker \alpha$ for all $a\in A, \alpha\in I$.

\item Let $\alpha\in\en(A)$. Then $f\alpha\in\en(A)$ if and only if for any  term $t$ 
\[t(x_1f,\ldots, x_nf)\alpha = t(x_1,\ldots, x_n)f\alpha .\]

\item If $I$ is a left ideal of $\en(A)$ with $I^2=I$, then 
$\lambda_f$ is a left translation of $I$ if and only if for all $\alpha\in I$ and for all terms $t$ 
\[t(x_1f,\ldots, x_nf)\alpha = t(x_1,\ldots, x_n)f\alpha. \]
\end{enumerate}
\end{lemma}

\begin{proof} (1) We have that  $\rho_f = \rho_g$ if and only if $\alpha f=\alpha g$ for every $\alpha\in I$. This is equivalent to 
$f|_{\im\alpha} = g|_{\im\alpha}$ for all $\alpha \in I$.
	
(2) Suppose $\alpha f\in\en (A)$ for some  $\alpha \in \en(A)$. Then for any term $t(a_1 \alpha, \ldots, a_n \alpha)$ where $a_1, \ldots, a_n \in A$ and $\alpha \in I$ we have
$$t(a_1 \alpha, \ldots, a_n \alpha)f = t(a_1, \ldots, a_n)\alpha f = t(a_1\alpha f , \ldots, a_n \alpha f),$$
showing that $f$ restricted to the image of $\alpha$ is a morphism. The converse is clear.

(3) One direction has been proved. Suppose then that $I$ is a right ideal of $\en(A)$ with $I^2=I$ and that $f$  restricted to the image of $\alpha$ is a morphism for all $\alpha \in I$. By (2) this is equivalent to $\alpha f\in \en(A)$ for all $\alpha\in I$.  Let $\alpha \in I$ and write
$\alpha=\beta\gamma$ where $\beta,\gamma\in I$. Then  
$$\alpha\rho_f = (\beta \gamma)\rho_f = (\beta\gamma) f=\beta(\gamma f) \in I,$$
since $\beta \in I$, $\gamma f \in \en(A)$ and $I$ is a right ideal of $\en(A)$.

(4) We have that $\lambda_f=\lambda_g$ if and only if $f\alpha= g\alpha$ for all $\alpha\in I$. This is clearly equivalent to the condition that
$(af,ag)\in \ker \alpha$ for all $a\in A, \alpha\in I$. 
	
(5) Let $\alpha\in\en(A)$. Then $f\alpha\in \en(A)$ if and only if  for all terms $t(x_1,\dots, x_n)$ 
\[t(x_1f,\ldots, x_nf)\alpha = t(x_1f\alpha,\ldots, x_nf\alpha)=t(x_1,\ldots, x_n)f\alpha,\]
the first equality holding 
since $\alpha$ is a morphism.

(6) This is dual to the argument in (3).
\end{proof}

For convenience, we state the following, which is readily verified.

\begin{proposition}\label{prop:realise} Let $I$ be a subsemigroup of $\en(A)$, let $S:=S(A, I)$ denote the idealiser of $I$ in $\en(A)$ and let $T:=T(A, I)$ denote the idealiser of $I$ in $\fullTA$.  Then 
\[\OmegaS(I)=\{ (\lambda_f,\rho_f):f\in S(A,I)\}\mbox{ and }\OmegaS[T](I)=\{ (\lambda_f,\rho_f):f\in T(A,I)\}\]
are subsemigroups of $\Omega(I)$ with $\OmegaS(I)\subseteq \OmegaS[T](I)$. 
\end{proposition}

We wish to determine when $\Omega(I)=\OmegaS(I)$ or $\Omega(I)=\OmegaS[T](I)$ and when the morphisms $\chiO^S$ or  $\chiO^T$ are isomorphisms. 

\subsection{Image and kernel conditions and reductivity}
In this subsection we consider two very natural notions on a pair $(A, I)$ where $A$ is an algebra  and $I$ is a subsemigroup of $\en(A)$ relating to the images and kernels of the maps in $I$. We introduce these conditions below, and explain how they relate to existing notions of reductivity.

\begin{definition}\label{defn:i}\label{defn:ikappa} Let $A$ be an algebra and let $I$ be a subsemigroup of $\mathcal{T}_A$. We define
\[\im I=\bigcup_{\alpha\in I}\im\alpha, \mbox{ and } \ker I= \bigcap_{\alpha\in I}\ker\alpha.\] 

The relations $\equiv_{\im}$ and $\equiv_{\ker}$ on $T(A, I)$ are then defined by
\[f \equiv_{\im} g\mbox{ if and only if }f|_{\im I} = g|_{\im I}\]
and
\[f \equiv_{\ker} g\mbox{ if and only if }(af, ag) \in \ker I\mbox{ for all }a \in A.\]
We denote the intersection of $\equiv_{\im}$ and $\equiv_{\ker}$ by $\equiv$.
\end{definition}

\begin{lemma}\label{lem:congruences}   Let $A$ be an algebra and let  $I$ be subsemigroup of $\en(A)$.  Then $\equiv_{\im}, \equiv_{\ker}$  and $\equiv$
are congruences on $T(A, I)$ and hence also on $S(A,I)$.\end{lemma}

\begin{proof} Suppose that $f, g, h \in T(A, I)$ and that $f \equiv_{\im} g$. It is clear from definition that $fh \equiv_{\im} gh$, since $(fh)|_{\im I} = f|_{\im I}h = g|_{\im I}h = (gh)|_{\im I}$. For each $b \in \im I$ there exists $a \in A$ and $\alpha \in I$ such that $b=a\alpha$. Since $h \in T(A, I)$ we have that $\alpha h \in I$, and hence also that $a\alpha h \in \im I$. Thus $bhf = a \alpha h f = a\alpha h g = bhg$, and since $b$ was arbitrary, we obtain $hf \equiv_{\im} hg$. A similar argument shows that $\equiv_{\ker}$ is also a congruence and so the result for $\equiv$ follows.\end{proof}
 
Let $I$ be a subsemigroup of $\en(A)$.
It follows easily from the definition of $\equiv_{\im}$ that if $\im I = A$, then $\equiv_{\im}$ is equality on both $T(A,I)$ and $S(A, I)$.  Dually, if $\ker I$ is the trivial relation, then $\equiv_{\ker}$ is equality on both $T(A,I)$ and $S(A, I)$.  Moreover, if $\im I$ generates $A$, that is,
$\I=A$, then $\equiv_{\im}$ is easily seen to be equality on $S(A,I)$, recalling that the elements of $S(A, I)$ are morphisms. This motivates the following definitions.

\begin{definition}\label{def:repandsep} Let $A$ be an algebra and $I$ a subsemigroup of  $\en(A)$. 
	\begin{enumerate}
		\item[$\bullet$] We say that $A$ is {\em $I$-representable} if $\I=A$.
		\item[$\bullet$] We say that  $A$ is {\em $I$-separable}  if $\ker I$ is trivial.
\end{enumerate}
\end{definition}
\begin{remark}\label{remark:strongrep}
Note that $I$-representability is a weak version of the \textit{weak transitivity} condition studied in \cite{Petrich:1970} (the latter meaning that $\im I$ is actually
equal to $A$, rather than merely generating $A$), while $I$-separability is referred to in \cite{Petrich:1970} as being {\em separative}.
\end{remark}
It is not hard to see that the $I$-representability and $I$-separability conditions are related to left/right reductivity. The notion of reductivity plays a significant role in earlier studies of translational hulls. We recall the following terminology \cite{Petrich:1970}, augmented for our purposes. 

\begin{definition}\label{defn:reductive}  A subsemigroup $I$ of $U$ is:
	\begin{enumerate}\item[$\bullet$] {\em left $U$-reductive} if for all $\alpha, \beta\in U$
		\[\gamma\alpha = \gamma\beta \mbox{ for all } \gamma \in I \mbox{ implies that }\alpha=\beta;\]
		\item[$\bullet$] {\em right  $U$-reductive} if for all $\alpha, \beta\in U$\[\alpha \gamma= \beta\gamma \mbox{ for all }\gamma \in I \mbox{ implies that } \alpha=\beta;\]
		\item[$\bullet$] {\em weakly $U$-reductive} if for all $\alpha, \beta\in U$ \[\gamma\alpha = \gamma\beta\mbox{ and }\alpha \gamma= \beta\gamma \mbox{ for all } \gamma \in I \mbox{ implies that } \alpha=\beta.\]
	\end{enumerate}
	
	The semigroup $I$ is  {\em left/right/weakly reductive} if it is left/right/weakly $I$-reductive (in which case the oversemigroup $U$ plays no role).  
\end{definition}

Let $I$ be a subsemigroup of a semigroup $U$. Note that it is immediate that if $I$ is left or right $U$-reductive, then it is weakly $U$-reductive. Further, if $I$ is left/right/weakly $U$-reductive
for some oversemigroup $U$, then it is left/right/weakly reductive. It is immediate from the definitions that if $I$ is an ideal
of $U$, then $\chiR^U$ [respectively,  $\chiL^U$] is a natural isomorphism if and only if $I$ is left [respectively,  right] $U$-reductive, and $\chiO^U$ is a natural isomorphism if and only if $I$ is weakly $U$-reductive. 
Further, from \cite{Petrich:1970} we have that if $I$ is left [respectively,  right] reductive, then $\piR$ [respectively,  $\piL$] is injective,  and so $\Omega(I)$ is naturally isomorphic to  $ \Rhot(I)$  [respectively,  $\Lambdat(I)$].

We now indicate how reductivity plays a role in the situation at hand, summarising the connections between the above conditions and the natural isomorphisms they impose.
 
\begin{theorem}\label{cor:rep}
\label{prop:quotient}\label{lem:reductive_quotients}	\label{lem:reductivity from rep or sep} Let $A$ be an algebra and let $I$ be a subsemigroup of $\en(A)$. Let $S:=S(A, I)$ denote the idealiser of $I$ in $\en(A)$ and let $T:=T(A, I)$ denote the idealiser of $I$ in $\fullTA$. Then 
	\[ \OmegaS[T](I)\cong T(A,I)/{\equiv} \quad\mbox{ and }\quad \OmegaS(I)\cong S(A,I)/{\equiv}.\]
	Further:
	\begin{enumerate}
		\item If $I$ is left reductive: $\OmegaS[T](I) \cong T(A, I)/{\equiv_{\im}}$, $\OmegaS(I) \cong S(A, I)/{\equiv_{\im}}$, and $\Rhot(I) \cong\Omega(I)$.
		\item If $I$ is right reductive: $\OmegaS[T](I) \cong T(A, I)/{\equiv_{\ker}}$, $\OmegaS(I) \cong S(A, I)/{\equiv_{\ker}}$, and $\Lambdat(I) \cong \Omega(I)$.
		\item If $U$  is any semigroup such that  $I \subseteq U \subseteq T(A, I)$, then $\equiv_{\im}$ is equality on $U$ if and only if $I$ is left $U$-reductive.
		In particular, if $I$ is left $T(A, I)$-reductive, then $\OmegaS[T](I) \cong T(A, I)$ and $\OmegaS(I) \cong S(A, I)$.
		\item  If $U$ is any semigroup such that  $I \subseteq U \subseteq T(A, I)$, then $\equiv_{\ker}$ is equality on $U$ if and only if $I$ is right $U$-reductive.
		In particular, if $I$ is right $T(A, I)$-reductive, then $\OmegaS[T](I) \cong T(A, I)$ and $\OmegaS(I) \cong S(A, I)$.
		\item If $A$ is $I$-representable, then $I$ is left $\en(A)$-reductive and hence in particular $\OmegaS[T](I) \cong T(A, I)/{\equiv_{\im}}$ and $\OmegaS(I) \cong S(A, I)$.
		\item $A$ is $I$-separable if and only if $I$ is right $\fullTA$-reductive.
		In this case $\OmegaS[T](I) \cong T(A, I)$ and $\OmegaS(I) \cong S(A, I)$.
	\end{enumerate}
	\end{theorem}

\begin{proof} The first statement is a matter of observing that
	\[(\lambda_f,\rho_f)=(\lambda_g,\rho_g)\mbox{ if and only if }f \equiv g.\]
	
	(1) Suppose that $I$ is left reductive and $f,g\in T$ are such that $f\equiv_{\im} g$. Then for any $\alpha\in I$ we have
	$\alpha f=\alpha g$. Now for any $\alpha, \beta \in I$ we certainly have
	\[\alpha (f\beta)=(\alpha f)\beta=(\alpha g)\beta=\alpha (g\beta).\]
	But $f\beta, g\beta\in I$ and $I$ is left reductive, so that $f\beta=g\beta$. Since this holds for any
	$\beta\in I$ we have $f\equiv_{\ker} g$. The result now follows from the definition of $\equiv$ and the comments preceding the theorem. 
	
	(2) This is dual to (1). 

(3) and (4) follow immediately from the definitions.

(5) Suppose that $A$ is $I$-representable. Let $\alpha, \beta \in \en(A)$ and suppose that $\gamma\alpha =\gamma\beta$ for all $\gamma \in I$. We aim to show that $\alpha = \beta$. Since $A$ is $I$-representable, we may express each $b \in A$ as $b=t(x_1\alpha_1, \ldots, x_k \alpha_k)$ for some term $t$, $x_i \in A$ and $\alpha_i \in I$. Then
	\begin{eqnarray*}
		b\alpha &=& t(x_1\alpha_1, \ldots, x_k \alpha_k)\alpha\\
		&=& t(x_1\alpha_1\alpha, \ldots, x_k \alpha_k\alpha)\\
		&=& t(x_1\alpha_1\beta, \ldots, x_k \alpha_k\beta)\\
		&=& t(x_1\alpha_1, \ldots, x_k \alpha_k)\beta\\
		&=&b\beta. 
	\end{eqnarray*}
	Since the above argument holds for all $b \in A$, we conclude that $\alpha=\beta$.
	
(6) First suppose that $A$ is $I$-separable. Let $\alpha, \beta \in \fullTA$ and suppose that $\alpha \gamma=\beta\gamma$ for all $\gamma \in I$. We aim to show that $\alpha = \beta$. Suppose not; then there exists $a \in A$ with $a\alpha \neq a\beta$. Since $A$ is $I$-separable there exists $\gamma \in I$ such that $a\alpha\gamma \neq a\beta\gamma$, contradicting the assumption that $\alpha\gamma = \beta\gamma$. Conversely, suppose that $I$ is right $\fullTA$-reductive. If there exist $a\neq b\in A$ such that $a\gamma = b\gamma$ for all $\gamma \in I$, then taking $\alpha$ to be the map sending all elements to $a$ and $\beta$ the map sending all elements to $b$ gives a contradiction. 

\end{proof}

\begin{remark}\label{remark:weaksep}
Part (5) of Theorem~\ref{cor:rep} says that $A$ being $I$-representable is a sufficient condition for $I$ be to left $\en(A)$-reductive. It is straightforward to show that the stronger condition of weak transitivity (see Remark~\ref{remark:strongrep} above) is in fact equivalent to $I$ being left $\T_A$-reductive.

On the other hand, part (6) of the theorem says that $A$ being $I$-separable is equivalent to $I$ being right $\T_A$-reductive. One can naturally define a notion of {\em weak $I$-separability}: namely $A$ is {\em weakly $I$-separable} if for all $\beta,\gamma\in\EndA$, we have that $(a\beta,a\gamma) \in \ker I$  implies  $a\beta=a\gamma$. It is easy to see
that this is a sufficient condition for $I$ to be right $\en(A)$-reductive.
\end{remark}

\section{Translations induced by transformations and endomorphisms}\label{sec:maps}

\subsection{Translations induced by transformations}\label{sub:maps}

We now introduce conditions for right  and left translations that are equivalent to being realised by mappings of $A$. 
These conditions are phrased in terms of  the behaviour of the translations themselves so are not as `pure' as one might wish.  Nevertheless, we put them to good use later in  applications to natural cases.

\begin{definition}\label{defn:balanced} Let $I$ be a subsemigroup of $\en(A)$.
\begin{enumerate}
	\item We say that a right translation $\rho \in \Rho(I)$  is {\em right-balanced} if 
	\[a\alpha=b\beta\mbox{ implies that }a(\alpha\rho)=b(\beta\rho)\]
	for all $a,b\in A$ and $\alpha,\beta\in I$.
	\item We say that a left translation $\lambda \in \Lambda(I)$ is {\em left-balanced} if 
	for all $a\in A$ there exists  $x_a\in A$ such that 
	\[a\lambda(\alpha)=x_a\alpha\mbox{ for all }\alpha\in I.\]
\end{enumerate}
\end{definition}

\begin{proposition}\label{prop:rtranslationsbymaps} Let $I$ be a subsemigroup of $\en(A)$and let $\rho\in \Rho(I)$ and $\lambda \in \Lambda(I)$.
	
\begin{enumerate}
\item   We have $\rho=\rho_f$ for some $f\in\fullTA$ if and only if $\rho$ is right-balanced.
\item We have $\lambda=\lambda_f$ for some $f\in\fullTA$ if and only if $\lambda$ is left-balanced.
\end{enumerate}
\end{proposition}
\begin{proof} 
(1) Clearly, if $\rho=\rho_f$ for some $f\in \fullTA$, then $\rho$ is right-balanced. 

For the converse, suppose that $\rho$ is right-balanced. Recall that $\im I=\bigcup_{\alpha\in I}\im \alpha$. Define $f:A\rightarrow A$ by 
\[xf=a(\alpha\rho) \mbox{ where }x=a\alpha\in \im I, \alpha\in I\]
and where $f$ takes any value for the elements of $A\setminus \im I$. 
 Since $\rho$ is right-balanced, $f$ is well-defined. By definition, $\rho=\rho_f$.
 
(2)  If $\lambda=\lambda_f$ for some $f\in \fullTA$, then putting $x_a=af$ we see that
$\lambda$ is left-balanced.  

Conversely, if $\lambda$ is left-balanced, then setting $af=x_a$ for all $a\in A$ yields $f\in\fullTA$ such that
$\lambda=\lambda_f$.\end{proof}

We note that (2) in Proposition \ref{prop:rtranslationsbymaps}  is essentially a matter of formalism. Nevertheless, the formulation of the notion of being left-balanced is occasionally useful.
\begin{theorem} \label{thm:maps} Let $I$ be a subsemigroup of $\en(A)$. Then
\[\Omega(I)=\{ (\lambda_f,\rho_f): f\in T(A,I)\}\]
if and only if every $\lambda\in \Lambdat(I)$ is left-balanced and every $\rho\in \Rhot(I)$ is right-balanced.
\end{theorem}
\begin{proof} 
Clearly, if $\Omega(I) =\{ (\lambda_f,\rho_f): f\in T(A,I)\}$, then by Proposition~\ref{prop:rtranslationsbymaps}, we immediately obtain that every right translation $\rho\in\Rhot(I)$ is right-balanced and every left translation $\lambda\in\Lambdat(I)$ is left-balanced.

Conversely, from Proposition~\ref{prop:realise} we know that 
$\{ (\lambda_f,\rho_f): f\in T(A,I)\}$ is a subsemigroup of $\Omega(I)$. It remains to show that every $(\lambda,\rho)\in \Omega(I)$
may be realised by some $f\in T(A,I)$.

Suppose that every $\lambda\in \widetilde{\Lambda}(I)$ is left-balanced and every $\rho\in \widetilde{\Rho}(I)$ is right-balanced.
Let $(\lambda,\rho)\in \Omega(I)$. We choose $f\in \fullTA$ as follows:
\[a f=x(\alpha\rho)\mbox{ for }a=x\alpha\in \im I, \alpha\in I\]
and for all $a\in A\setminus \im I$,
\[af=x_a \]
where $x_a\in A$ is chosen so that 
\[a\lambda(\alpha)=x_a\alpha\mbox{ for all }\alpha\in I.\]
 By the proof of Proposition \ref{prop:rtranslationsbymaps} we have that $\rho=\rho_f$. Let  $a\in A$. If  
$a=x\alpha$ for some $x\in A,\alpha\in I$, then for all $\beta\in I$,
\[a\lambda(\beta)=(x\alpha)\lambda(\beta)=x(\alpha \lambda(\beta))=x((\alpha\rho)\beta)=x(\alpha\rho)\beta= af\beta.\]
Otherwise,
\[a\lambda(\beta)=x_a\beta=af\beta.\]
Thus $\lambda(\beta)=f\beta$. Consequently,  $(\lambda,\rho)=(\lambda_f,\rho_f)$ and 
$\Omega(I)$ has the required form.\end{proof}

\subsection{Translations induced by morphisms}

We next consider when the translational hull of $I$ is realised by endomorphisms. To this end we make the first connections between $I$-representability
and $I$-separability and the notions of being left-balanced and right-balanced that guarantee realisation by mappings. 

\begin{lemma}\label{lem:lbalanced} Let $A$ be an algebra and let $I$ be a subsemigroup of $\en(A)$. If $A$ is $I$-representable and $(\lambda,\rho)\in\Omega(I)$, then $\lambda$ is left-balanced.
Specifically, if for  each $a\in A$ we choose and fix an expression 
$a=t(x_1\alpha_1,\dots,x_n\alpha_n)$ where $\alpha_i\in I$, $1\leq i\leq n$, then we can take
\[x_a=t(x_1(\alpha_1\rho),\dots,x_n(\alpha_n\rho)).\]
\end{lemma}
\begin{proof}  Let $a\in A$ and write $a=t(x_1\alpha_1,\dots,x_n\alpha_n)\in A$, where $\alpha_i\in I$, $1\leq i\leq n$, making use of $I$-representability.
Then for all $\alpha\in I$ we have 
\[a\lambda(\alpha)=t(x_1\alpha_1,\dots,x_n\alpha_n)\lambda(\alpha)=t(x_1\alpha_1\lambda(\alpha),\dots,x_n\alpha_n\lambda(\alpha))=t(x_1(\alpha_1\rho),\dots,x_n(\alpha_n\rho))\alpha,\]
so the result holds with $x_a=t(x_1(\alpha_1\rho),\dots,x_n(\alpha_n\rho))$.
\end{proof}

\begin{lemma}\label{lem:balanced} Let $A$ be an algebra and let $I$ be a subsemigroup of $\en(A)$. If $A$ is $I$-separable and $(\lambda,\rho)\in\Omega(I)$, then $\rho$ is right-balanced.
\end{lemma}
\begin{proof} Suppose that $a\alpha=b\beta$ for some $a,b\in A$ and $\alpha,\beta\in I$. Then for all $\gamma\in I$ we have
\[a(\alpha\rho)\gamma=a\alpha\lambda(\gamma)=b\beta\lambda(\gamma)=
b(\beta\rho)\gamma\]
so that $I$-separability gives $a\alpha\rho=b\beta\rho$.
\end{proof}

We now show that combining the results of Subsection ~\ref{sub:maps} with Lemma~\ref{lem:lbalanced} provides sufficient conditions on an algebra $A$ and subsemigroup $I$ of $\en(A)$ in order for the translational hull to be realised by mappings. In Section~\ref{sec:free algebras} we shall see that these conditions are satisfied in a number of interesting cases.
\begin{proposition}
	\label{prop:suffmaps}
	Let $A$ be an algebra with generating set $X$ and let $I$ be a subsemigroup of $\en(A)$ such that $I^2=I$. Suppose that for each $c \in \im I:=\bigcup_{\gamma \in I} \im \gamma$ there exists a morphism $\gamma_c \in I$ with the property that $x \gamma_c = c$ for all $x \in X$. Then:
	\begin{enumerate}
		\item every right translation of $I$ is right-balanced;
		\item if $A$ is $I$-representable then \[\Omega(I) = \OmegaS[T](I)= \{(\lambda_f, \rho_f): f \in T(A, I)\} \cong T(A, I)/{\equiv_{\im}}.\]
	\end{enumerate} 
\end{proposition}
\begin{proof} 		
	(1) We show that any $\rho\in \Rho(I)$ is right-balanced. Let $a,b\in A$ and $\alpha,\beta\in I$ with $a\alpha=b\beta$. Since $I=I^2$ there exist $\gamma_1, \gamma_2, \delta_1, \delta_2 \in I$ such that $\alpha = \gamma_1\gamma_2$ and $\beta = \delta_1\delta_2$. Set $c=a\gamma_1$ and $d=b\delta_1$. Then $c, d \in \im I$, and by assumption there exists $\gamma_c, \gamma_d \in I$ such that $x\gamma_c = c$ and $x\gamma_d = d$ for all $x \in X$. It then follows easily that  $\gamma_c\gamma_2=\gamma_d\delta_2$. 
	Now \[\gamma_c(\gamma_2\rho)=(\gamma_c\gamma_2)\rho=(\gamma_d\delta_2)\rho=\gamma_d(\delta_2\rho).\]
	Thus 
	\begin{align*}
a(\alpha\rho)&=a(\gamma_1\gamma_2\rho) = a \gamma_1(\gamma_2\rho) = c(\gamma_2\rho) \\
&= x\gamma_c(\gamma_2\rho)=x\gamma_d(\delta_2\rho)=d(\delta_2\rho)\\ 
&=b\delta_1(\delta_2\rho)=b(\delta_1\delta_2\rho)=b(\beta\rho)
\end{align*}
	where $x$ is taken to be any element of $X$, hence showing that $\rho$ is right-balanced.
	
(2) If $A$ is $I$-representable then Lemma \ref{lem:lbalanced} gives that every linked left translation is left-balanced and the result then follows from Theorem \ref{thm:maps} and part (6) of Theorem \ref{lem:reductive_quotients}.
\end{proof}

The following result is also now immediate from Theorems \ref{cor:rep} and \ref{thm:maps},  Lemmas  \ref{lem:lbalanced} and \ref{lem:balanced}.

\begin{theorem}
Let $A$ be an algebra and $I$ a subsemigroup of $\en(A)$. If $A$ is both $I$-representable and $I$-separable then 
$\Omega(I) = \{(\lambda_f, \rho_f): f \in T(A, I)\}\cong T(A,I)$.
\end{theorem}

The interaction of $I$-representability and $I$-separability is very subtle, and it is together that they have the most power. We shall soon see (in Theorem~\ref{thm:main} below) that when both conditions hold we have the stronger result that every bi-translation of $I$ is realised by \emph{endomorphisms} of $A$, that is, $\Omega(I) = \{(\lambda_f, \rho_f): f \in S(A, I)\}$, (and moreover, $\Omega(I)$ is isomorphic to $S(A, I)$ in this case). 
\begin{remark}\label{rem:repandsep} 
In Subsection~\ref{sub:maps} we considered transformations of $A$. Writing as before $\I$ for the algebra generated by all images of maps in $I$ it is clear that each $\alpha \in I$ (including $\alpha\rho$ for each $\alpha \in I$ and $\rho\in P(I)$) can be viewed as a map $\alpha : A \rightarrow \I$. Hence, if $f \in \fullTA[\I]$ the composition $\alpha f$ makes sense, and moreover if $f$ is a morphism then so is $\alpha f$. If $f \in \fullTA[\I]$ has the property that $\alpha f \in I$ for all $\alpha \in I$ we shall, by a slight abuse of notation, write $\rho_f$ to denote the right translation given by $\alpha \rightarrow \alpha f$. We shall show that when the $I$-separability condition holds, for each linked right translation $\rho \in \Rhot(I)$  we can construct a well-defined \textit{endomorphism} $f$ of $\I$ that encodes the behaviour of the right translation via $\rho = \rho_f$; moreover, there is {\em no choice} for what this endomorphism should be. Then, in the case where the $I$-representability condition also holds, so that $A=\I$,  in this case the endomorphism constructed will be in the idealiser $S$ of $I$. Towards this goal, we make the following two definitions.
\end{remark}

\begin{definition}\label{defn:sbalanced} Let $I$ be a subsemigroup of $\en(A)$.
	
\begin{enumerate} 
\item We say that a right translation $\rho\in \Rho(I)$ is {\em strongly  right-balanced} if for all \[t(x_1\alpha_1,\ldots ,x_m\alpha_m)=
s(y_1\beta_1,\ldots ,y_n\beta_n)\in \I\] where $\alpha_i, \beta_j\in I$ and $x_i, y_j \in A$ for $1\leq i\leq m, 1\leq j\leq n$, we have
\[t(x_1\alpha_1\rho,\ldots ,x_m\alpha_m\rho)=
s(y_1\beta_1\rho,\ldots ,y_n\beta_n\rho).\]
\item  We say that a left translation $\lambda \in \Lambda(I)$ is {\em strongly left-balanced} if 
for all $a\in A$ there exists  $x_a\in A$ such that 
\[a\lambda(\alpha)=x_a\alpha\mbox{ for all }\alpha\in I\]
and the $x_a$ {\em can be chosen} so that the map $a\mapsto x_a$ is a morphism.
\end{enumerate}
\end{definition}

\begin{proposition}\label{prop:rtranslationsbymorphisms} Let $I$ be a subsemigroup of $\en(A)$ and let  $\rho \in \Rho(I)$.  There exists a morphism $f:\I\rightarrow \I$  such that $\rho=\rho_f$ if and only if $\rho$ is strongly right-balanced. Moreover, for all
	\[a=t(x_1\alpha_1,\ldots ,x_m\alpha_m)\in \I\]
	where $\alpha_i\in I$ and $x_i\in A$, for  $1\leq i\leq m$, we have that $af$ must be given by
	\[af=t(x_1\alpha_1\rho,\ldots ,x_m\alpha_m\rho).\]

Consequently,  if $A$ is $I$-representable, then  $\rho=\rho_f$ for some $f\in\en(A)$ if and only if $\rho$ is strongly right-balanced, and in this case
$f$ must be defined as given.
\end{proposition}
\begin{proof}
Suppose first that $\rho=\rho_f$ for some morphism $f:\I \rightarrow \I$. Let $t(x_1\alpha_1,\ldots ,x_m\alpha_m)=
s(y_1\beta_1,\ldots ,y_n\beta_n)\in\I$. Then 
\[\begin{array}{rcl}
t(x_1\alpha_1\rho,\ldots ,x_m\alpha_m\rho)&=&t(x_1\alpha_1f,\ldots ,x_m\alpha_mf)\\
&=&t(x_1\alpha_1,\ldots ,x_m\alpha_m)f\\
&=&s(y_1\beta_1,\ldots ,y_n\beta_n)f\\
&=&s(y_1\beta_1f,\ldots ,y_n\beta_nf)\\
&=&s(y_1\beta_1\rho,\ldots ,y_n\beta_n\rho),
\end{array}\]
so that $\rho$ is strongly right-balanced.

Conversely, if $\rho$ is strongly right-balanced, then there is a well-defined  map $f:\I \rightarrow \I$ given by 
$af=t(x_1\alpha_1\rho,\ldots ,x_m\alpha_m\rho)$, where $a=t(x_1\alpha_1,\ldots ,x_m\alpha_m)$ {\em for any}
such expression of $a$. It is easy to check from this formulation that $\rho = \rho_f$ and $f\in \en(\I)$.

Finally, if $A$ is $I$-representable then $A=\I$ and we immediately see that $\rho=\rho_f$ for some (the given) $f \in \en(A)$ if and only if $\rho $ is strongly right-balanced.
\end{proof}

\begin{proposition}\label{prop:ltranslationsbymorphisms} Let $I$ be a subsemigroup of $\en(A)$ and let $\lambda\in \Lambda(I)$. 
	Then $\lambda=\lambda_f$ for some $f\in\en(A)$  if and only if $\lambda$ is {\em strongly left-balanced}.

	Further, if $A$ is $I$-separable and $(\lambda,\rho)\in \Omega(I)$, then $\lambda=\lambda_f$  for some $f\in\en(A)$ if and only if $\lambda$ is  left-balanced.
\end{proposition}
\begin{proof} The first statement follows  from the definitions. 
	
For the final statement suppose that $A$ is $I$-separable and let $(\lambda,\rho)\in \Omega(I)$.
By Proposition~\ref{prop:rtranslationsbymaps}, $\lambda$ is left-balanced if and only if $\lambda=\lambda_f$ for some $f\in\fullTA$.  We show that the $I$-separability condition forces $f$ to be a morphism.
	For any $t(x_1,\dots,x_n)\in A$ and $\alpha\in I$ we have
	\[\begin{array}{rcl}
	t(x_1,\dots,x_n)f\alpha&=&t(x_1,\dots,x_n)\lambda(\alpha)\\
	&=&t(x_1\lambda(\alpha),\dots,x_n\lambda(\alpha))\\
	&=&t(x_1f\alpha,\dots,x_nf\alpha)\\
	&=&t(x_1f,\dots,x_nf)\alpha.
	\end{array}\]
	As this is true for all $\alpha\in I$, we deduce from $I$-separability of $A$ that $t(x_1,\dots,x_n)f=t(x_1f,\dots,x_nf)$ so that $f\in\en(A)$, as
	desired.
\end{proof} 

We can put together the results of Propositions~\ref{prop:rtranslationsbymorphisms} and \ref{prop:ltranslationsbymorphisms}  to obtain an
analogue of Theorem~\ref{thm:maps}. 

\begin{theorem}\label{cor:morphisms} Let $I$ be a subsemigroup of $\en(A)$. Then
	\[\Omega(I)=\{ (\lambda_f,\rho_f): f\in S(A,I)\}\]
	if and only if for all $(\lambda,\rho)\in \Omega(I)$ we have that
	$\rho$ is strongly right-balanced and $\lambda$ is strongly left-balanced, with 
	\[x_{t(x_1\alpha_1,\dots,x_n\alpha_n)}=t(x_1\alpha_1\rho,\dots,x_n\alpha_n\rho)\]
	for all $t(x_1\alpha_1,\dots,x_n\alpha_n)\in \I$. 
\end{theorem}
\begin{proof} 
Suppose first that  $\Omega(I)=\{ (\lambda_f,\rho_f): f\in S(A, I)\}$. 
Let $(\lambda,\rho) \in \Omega(I)$. By assumption, we have that there exists $f\in\EndA$ such that $\lambda= \lambda_f$ and $\rho=\rho_f$. Notice that for each $b \in \im I$ we have $b= a \alpha$ for some $a \in A$, $\alpha \in I$, and so $bf = a\alpha f = a\alpha\rho \in \im I$, from which it follows easily that $f|_{\I}\in\EndA[\I]$. Since $f \in \en(A)$ we have $\lambda$ is strongly left-balanced by Proposition~\ref{prop:ltranslationsbymorphisms}, and noting that $\rho_f=\rho_{f|_{\I}}$, we get that $\rho$ is strongly right-balanced by Proposition~\ref{prop:rtranslationsbymorphisms}.
Moreover, for any $a=t(x_1\alpha_1,\dots,x_n\alpha_n)\in \I$, we have that 
\[af=t(x_1\alpha_1,\dots,x_n\alpha_n)f=t(x_1\alpha_1f,\dots,x_n\alpha_nf)=t(x_1\alpha_1\rho,\dots,x_n\alpha_n\rho),\]
using Proposition~\ref{prop:rtranslationsbymorphisms} again, so that the $x_a$ may be chosen as indicated.

Conversely, suppose that $(\lambda,\rho) \in \Omega(I)$ with $\lambda$ strongly left-balanced and $\rho$ strongly right-balanced with the $x_a$ given as above for all $a\in\I$.
Then by Propositions~\ref{prop:rtranslationsbymorphisms} and \ref{prop:ltranslationsbymorphisms}, there exist $f\in\EndA$ and $g\in\EndA[\I]$ such that $\lambda = \lambda_f$ and $\rho = \rho_g$.
It remains to show that $\rho = \rho_f$. 

Recalling that $x_a = af$, 
this follows from the fact that for all $a = t(x_1\alpha_1,\dots,x_n\alpha_n)\in \I$ we have
$$ ag = t(x_1\alpha_1,\dots,x_n\alpha_n)g = t(x_1\alpha_1 g,\dots,x_n\alpha_n g) = t(x_1\alpha_1\rho,\dots,x_n\alpha_n\rho) = x_{a} = af, $$
and thus $g = f|_{\I}$, giving that $\rho = \rho_{f|_{\I}} = \rho_f$.  In particular, we have that $f\in S(A,I)$.

\end{proof} 

The previous result gives a characterisation (via properties of linked left and right translations) for the translational hull to be realised by endomorphisms. We conclude this section by demonstrating that $I$-representability and $I$-separability together provide a sufficient condition (via properties of the endomorphisms in $I$) for this to hold.

To state our next result, it is convenient to have a slight variation of the $I$-separability condition.
\begin{definition}
Let $C$ be a subalgebra of $A$, and let $\ker(C, I)$ be the relation on $C$ defined by 
\[\ker(C, I)=\{(a,b)\in C\times C: a\gamma = b\gamma \; \mbox{for all } \gamma \in I\}.\]
We say that {\em $C$ is $I$-separable} if $\ker(C, I)$ is trivial.
\end{definition}

\begin{proposition}\label{prop:BSEP}
Let $I$ be a subsemigroup of $\en(A)$ and let $\I$  denote the subalgebra of $A$ generated by $\bigcup_{\alpha \in I} \im\alpha$. Suppose that $\I$ is $I$-separable.

\begin{enumerate}
	\item Every linked right translation of $I$ is strongly right-balanced, that is
	$$\Rhot(I) \subseteq \{\rho_f: f \in \en{\I}\}.$$
\item If $\rho_f \in \Rhot(I)$  for some $ f \in \en(\I)$, then the left translations $\lambda$ linked to $\rho_f$ are those satisfying $\lambda(\alpha)|_{\lowerme{\I}} = (f\alpha)|_{\lowerme{\I}}$ for all $\alpha \in I$.
\end{enumerate}  
\end{proposition}

\begin{proof}
(1) Suppose that $\I$ is $I$-separable. Let $(\lambda,\rho)\in \Omega(I)$. We show that $\rho$ is strongly right-balanced. To this end suppose that 
$t(x_1\alpha_1,\ldots ,x_m\alpha_m)=
s(y_1\beta_1,\ldots ,y_n\beta_n)\in A$. Then for all $\alpha\in I$ we have
\[\begin{array}{rcl}
t(x_1\alpha_1\rho,\ldots ,x_m\alpha_m\rho)\alpha&=t(x_1(\alpha_1\rho)\alpha,\ldots ,x_m(\alpha_m\rho)\alpha)&\\&=t(x_1\alpha_1\lambda(\alpha),\ldots ,x_m\alpha_m\lambda(\alpha))
&=t(x_1\alpha_1,\ldots ,x_m\alpha_m)\lambda(\alpha)\\
&=s(y_1\beta_1,\ldots ,y_n\beta_n)\lambda(\alpha)\qquad \;\; &=s(y_1\beta_1\lambda(\alpha),\ldots ,y_n\beta_n\lambda(\alpha))\\
&=s(y_1(\beta_1\rho)\alpha,\ldots ,y_n(\beta_n\rho)\alpha)\;\;\,&=s(y_1(\beta_1\rho),\ldots ,y_n(\beta_n\rho))\alpha,\end{array}\]
and since this is true for all $\alpha\in I$, the fact that $\I$ is $I$-separable gives 
\[t(x_1\alpha_1\rho,\ldots ,x_m\alpha_m\rho)=s(y_1(\beta_1\rho),\ldots ,y_n(\beta_n\rho))\]
so that $\rho$ is strongly right-balanced. 

(2) By the previous part each element of  $\Rhot(I)$ is strongly right-balanced, and so by Proposition \ref{prop:rtranslationsbymorphisms} is of the form $\rho_f$ for some  $f \in \en(\I)$. Suppose that $(\lambda, \rho_f) \in \Omega(I)$. For all $t(x_1\alpha_1, \ldots, x_m \alpha_m) \in \I$ and all $\alpha \in I$ we clearly have
\begin{eqnarray*}
t(x_1\alpha_1, \ldots, x_m \alpha_m)\lambda(\alpha) &=& t(x_1\alpha_1\lambda(\alpha) , \ldots, x_m \alpha_m\lambda(\alpha) ) = t(x_1(\alpha_1\rho_f)\alpha , \ldots, x_m (\alpha_m\rho_f)\alpha)\\
&=& t(x_1\alpha_1f\alpha , \ldots, x_m \alpha_mf\alpha) = t(x_1\alpha_1, \ldots, x_m \alpha_m)f\alpha, 
\end{eqnarray*}
demonstrating that each left translation $\lambda$ linked to $\rho_f$ satisfies $\lambda(\alpha)|_{\I} = (f\alpha)|_{\I}$ for all $\alpha \in I$. Conversely, suppose that $\lambda$ is a left translation satisfying $\lambda(\alpha)|_{\I} = (f\alpha)|_{\I}$ for all $\alpha \in I$. Then for all $\alpha, \beta \in I$ and all $a \in A$ we have $a(\alpha\rho_f)\beta = a(\alpha f) \beta = (a\alpha)(f \beta) = (a\alpha)\lambda(\beta)$, demonstrating that $\lambda$ is linked to $\rho_f$.
\end{proof}

The next result is a strengthening, in several directions, of a result of \cite{Gluskin:1959}, available as \cite[Corollary 1, p. 312]{Petrich:1970}. 

\begin{theorem}

	\label{thm:main}
	Let $A$ be an algebra and $I$ a subsemigroup of $\en(A)$ such that $A$ is both $I$-representable and $I$-separable. Then
	with $S=S(A,I)$, the idealiser of $I$ in $\en(A)$, we have:
	\[
	\Rhot(I) = \{\rho_f: f \in S\},\quad
	\Lambdat(I)= \{\lambda_f: f \in S\}\quad \text{and} \quad
	\Omega(I) =\{(\lambda_f, \rho_f): f \in S\}.\]
	
	In particular,
	\[\Rhot(I)  \cong \Omega(I)\cong \Lambdat(I) \quad \mbox{ and } \quad \Rhot(I)  \cong S\cong \Lambdat(I),\] via natural isomorphisms.

\end{theorem}

\begin{proof}
Suppose that $A$ is both $I$-representable and $I$-separable, and let $(\lambda,\rho)\in \Omega(I)$. Since $A$ is $I$-representable we have $\I=A$, and since $A$ is $I$-separable Proposition \ref{prop:BSEP} gives that $\rho=\rho_f$ for some $f \in \en(A)$, and moreover (using again the fact that $\I=A$) that $\lambda=\lambda_f$. 
Hence $f\in S$ and we see that each of the stated equalities holds. In light of Theorem~\ref{lem:reductivity from rep or sep}, noting in particular that $\Rhot(I)=\RhoS(I)$ and $\Lambdat(I) =\LambdaS(I)$ here, the natural isomorphisms exist as given.  
\end{proof}

\begin{remark}
It is clear from the definitions that for any algebra $A$, if $I$ is taken to be equal to $\rm {End}(A)$ (or indeed, any subsemigroup containing the identity morphism), then $A$ is both $I$-representable and $I$-separable. It is also easy to see that if $A$ is $I$-representable (respectively, $I$-separable) and $J$ is a subsemigroup of $\en(A)$ satisfying $I \subseteq J \subseteq \en(A)$, then the pair $A$ is also $J$-representable (respectively, $J$-separable). Hence we pose the following questions. 
\end{remark}
\begin{question}\label{qn:min} For a given algebra $A$:
\begin{itemize}
	\item What are the minimal (with respect to inclusion) subsemigroups $I$ of $\en(A)$ such that $\en{A}$ is both $I$-representable and $I$-separable?
	\item Do there exist subsemigroups $I$ of $\en(A)$ with the property that the bi-translations are realised by endomorphisms of $A$, yet $A$ is not both
	$I$-representable and $I$-separable?	
\end{itemize}
\end{question}

\section{Representability without separability}\label{sec:nosep}
Throughout this section let $A$ be an algebra and let $I$ be a subsemigroup of $\EndA$. Our aim  is to define a congruence $\sim$ on $A$ and a congruence
$\approx$ on $I$ so that $I/{\approx}$ is a subsemigroup of $\en(A/{\sim})$, and such that the quotient $A/{\sim}$ is $I/{\approx}\,$-separable, and is also $I/{\approx}\,$-representable in the case that $A$ is $I$-representable. We then proceed to show how in certain circumstances
we can deduce information for $\Omega(I)$ from that of $\Omega(I/{\approx})$. 

\begin{definition}\label{defn:sim}  Let $k\in \mathbb{N}$ and let $\sim_k$ denote the relation on $A$ defined by 
	\[x \sim_k y\mbox{ if and only if } x\alpha=y\alpha\mbox{ for all }\alpha\in I^k.\]
Furthermore, let ${\sim}=\bigcup_{k\in \mathbb{N}}\sim_k$.
\end{definition}

It is clear that each $\sim_k$ is an equivalence relation on $A$.

\begin{lemma}\label{lem:start}
The relations $\sim_k$ where $k \in \mathbb{N}$ and the relation $\sim$ are all congruences on $A$. 
\end{lemma}
\begin{proof}
To see that $\sim_k$ is a congruence, 
	suppose that $a_i\sim_k b_i$ for $i=1, \ldots, n$. Then for any $n$-ary term $t$ and any $\alpha \in I^k$ we have:
	$$t(a_1, \ldots, a_n) \alpha= t(a_1\alpha, \ldots, a_n\alpha) =  t(b_1\alpha, \ldots, b_n\alpha) = t(b_1, \ldots, b_n) \alpha$$
	so that $t(a_1, \ldots, a_n) \sim_k t(b_1, \ldots, b_n)$.
	
Finally notice that 
	\[{\sim_1} \subseteq {\sim_2} \subseteq \cdots\subseteq {\sim}\]
	so that $\sim$ (being the union of a chain of congruences) is a congruence.
\end{proof}

In what follows we denote the $\sim_k$ and $\sim$-classes of $a\in A$ by $[a]_k$ and $[a]$, respectively. 

\begin{lemma}\label{lem:morphismsagain}
	 
	\begin{enumerate} \item For all $k$ the  semigroup  $I$ acts on $A/{\sim_k}$ by morphisms, under $[a]_k\alpha=[a\alpha]_k$.
		\item The semigroup  $I$ acts on $A/{\sim}$ by morphisms, under $[a]\alpha=[a\alpha]$.
		\item If ${\sim_{k}} = {\sim_{k+1}}$ for some $k\in \mathbb{N}$, then ${\sim_k} = {\sim_{k+n}} = {\sim}$, for every $n\in\mathbb{N}$.
		\item We have ${\sim_{k}} = {\sim_{k+1}}$ if and only if $[a]_k\alpha=[b]_k\alpha$ for all $\alpha\in I$ implies that $[a]_k=[b]_k$. 
		\item $A$ is $I$-separable if and only if $\sim$ is equality.
	\end{enumerate} 
\end{lemma} 
\begin{proof} (1) Let $a,b\in A$ with $a\,\sim_1\, b$. Then $a\alpha=b\alpha$ for all $\alpha\in I$ so  that certainly the action is well-defined. Suppose then that $\alpha\in I$ and $a\sim_k b$ for some $k >1$. Then  $a\alpha\gamma=b\alpha\gamma$ for all $\gamma\in I^{k-1}\supseteq I^k$. It is then clear that $a\alpha\sim_k b\alpha$ so that the action is again well defined. It is then routine to check that it is an action by morphisms.
	
	(2) Suppose that $\alpha\in I$ and $a\sim b$. Then $a\sim_k b$ for some $k\in\mathbb{N}$, so that 
	as above  $a\alpha\sim_k b\alpha$ and hence $a\alpha\sim b\alpha$ so that the action is well defined. Again, one sees that  $I$ acts by morphisms.
	
	(3)  Suppose that for some  $k$ we have ${\sim_k} ={\sim_{k+1}}$. It suffices to show that ${\sim_{k+1}}={\sim_{k+2}}$.
	To this end, let $a,b\in A$ with $a\sim_{k+2}b$ and $\alpha\in I$. Then for
	any $\gamma\in I^{k+1}$ we have that $a\alpha\gamma=b\alpha\gamma$ so that $a\alpha\sim_{k+1}b\alpha$.
	By assumption, $a\alpha\sim_{k} b\alpha$ so that $a\alpha\delta=b\alpha\delta$ for any $\delta\in I^k$. As this holds for any
	$\alpha\in I$, we deduce that $a \sim_{k+1} b$. 
	
	(4) We have the following series of equivalent statements:
	\[\begin{array}{rcll}
	{\sim_{k}} = {\sim_{k+1}} &\Leftrightarrow& a\sim_{k+1} b\mbox{ implies that }a\sim_k b\\
	&\Leftrightarrow& a\alpha=b\alpha\mbox{ for all }\alpha\in I^{k+1} \mbox{ implies that }a\alpha=b\alpha\mbox{ for all }\alpha\in I^{k}\\
	&\Leftrightarrow& [a]_k\beta=[b]_k\beta\mbox{ for all }\beta\in I \mbox{ implies that }[a]_k=[b]_k.
	\end{array}\]
	
	(5) If $\sim$ is equality, and $a\alpha=b\alpha$ for all $a,b\in A$, then certainly $a\sim_1 b$ so that $a\sim b$ and hence $a=b$. Thus $A$ is $I$-separable.
	
	For the converse, suppose $A$ is $I$-separable and $a \sim b$. Then $a \sim_k b$ for some $k\in \mathbb{N}$ and so 
	$a\gamma=b\gamma$ for all $\gamma\in I^k$.  If $k>1$, then taking any $\alpha\in I^{k-1}$ and $\beta\in I$, we
	see that $a\alpha\beta=b\alpha\beta$, so that using $I$-separability, $a\alpha=b\alpha$. Since this is true for {\em any} $\alpha \in I^{k-1}$ this yields that $a \sim_{k-1} b$ and by finite induction, that
	$a \sim_1 b$. Again by $I$-separability, we deduce that $a=b$.
\end{proof}

\begin{definition}\label{defn:eventual} We say $A$ is {\em eventually $I$-separable} if ${\sim_k} ={\sim_{k+1}}$ for some
	$k\in \mathbb{N}$. In this case, we let $k(I)$ be the least such $k$ for which this occurs, so that ${\sim}={ \sim_{k(I)}}$.
\end{definition}

Clearly, if $I$ is periodic in the power semigroup of $\EndA$, that is, $I^h=I^{h+\ell}$ for some $h,\ell\in\mathbb{N}$, then  $A$ is eventually $I$-separable. This is certainly the case if $I=I^2$, in which case $k(I)=1$. Examples of this behaviour include the situation where $I$ is regular (that is, for all $\alpha\in I$ there exists $\beta\in I$ such that $\alpha=\alpha\beta\alpha$) or, more generally, if every $\alpha\in I$ has a right or a left identity in $I$. We will see another  situation where $I=I^2$ in Section 
\ref{sec:free algebras}.

We have seen that $I$ acts on $A/{\sim}$, but there is no reason to suppose that it acts faithfully. 

\begin{definition}\label{defn:approx} Let $\approx$ denote the congruence on $I$ induced by the
	action of $I$ on $A/{\sim}$, that is, $\alpha \approx \beta$ if and only if $a\alpha \sim a \beta$ for all $a \in A$. We write $[[\alpha]]$ to denote the $\approx$-class
	containing an element $\alpha \in I$. 
\end{definition}

\begin{remark}
	Since $I/{\approx}$ acts faithfully on $A/{\sim}$ by morphisms, via $[a][[\alpha]]=[a]\alpha=[a\alpha]$, we can view $I/{\approx}$ as embedded in  $\en(A/{\sim})$.
\end{remark}

\begin{theorem}\label{prop:fixedrepandsep} Regarding $I/{\approx}$ as a subsemigroup of $\en(A/{\sim})$, 
the quotient algebra $A/{\sim}$ is $I/{\approx}\,$-separable. If $A$ is $I$-representable then $A/{\sim}$ is $I/{\approx}\,$-representable.
\end{theorem}
\begin{proof} Suppose that $[x], [y] \in A/{\sim}$ such that
	$[x][[\gamma]] = [y][[\gamma]]$ for every $[[\gamma]] \in I /{\approx}$. The latter by definition of the action means that $[x\gamma] = [y\gamma]$ for
	every $\gamma \in I$, in other words,
	$x \gamma \sim y \gamma$ for every $\gamma \in I$. By the definition of $\sim$ this means that there is a $k$ such that $x\gamma\delta = y\gamma\delta$ for
	all $\gamma \in I$ and $\delta \in I^k$. But this means exactly that $x \alpha = y \alpha$ for every $\alpha \in I^{k+1}$, which implies that $x \sim y$,
	so that $[x]=[y]$. 
	
If $A$ is $I$-representable then we may express each $x \in A$ as $x=t(a_1\alpha_1,\ldots, a_n\alpha_n)$ for some term $t$, elements $a_i\in A$ and $\alpha_i\in I$, for $1\leq i\leq n$. It is then clear that 
	$[x]=t([a_1][[\alpha_1]],\ldots, [a_n][[\alpha_n]])$ in $A/{\sim}$, where $[a_i]\in A/{\sim}$ and $[[\alpha_i]]\in I/{\approx}$, so that $A/{\sim}$ is $I/{\approx}\,$-representable.
\end{proof} 

Combining with Theorems~\ref{cor:rep} and \ref{thm:main} we obtain:
\begin{theorem}\label{cor:hom} Suppose that $A$ is $I$-representable. Then
	the translational hull $\Omega(I/{\approx})$ is (canonically) isomorphic to the idealiser of $I/{\approx}$ in $\en(A/{\sim})$.
\end{theorem}

We now show that under some natural conditions,  $\Omega(I/{\approx})$ is a morphic image of $\Omega(I)$. 

\begin{lemma}\label{lem:nicedescription} Suppose that 
	$A$ is finitely generated or eventually $I$-separable.   Then   $\alpha \approx \beta$ if and only if  $\alpha\gamma=\beta\gamma$ for 
	some $k\in \mathbb{N}$ and all $\gamma \in I^k$.
\end{lemma}

\begin{proof} Suppose first that  $k$ exists such that
	$\alpha\gamma = \beta\gamma$ for all $\gamma \in I^k$.  Then for every $a \in A$ we have $(a\alpha)\gamma = (a\beta)\gamma$ for all $\gamma \in I^k$
	which means $a\alpha \sim a\beta$ for all $a \in A$ and so $\alpha \approx \beta$.
	
	Conversely, suppose $\alpha \approx \beta$ so that for all $a \in A$ we have $a\alpha \sim a\beta$. Let $k_a\in\mathbb{N}$ be chosen such that
	$a\alpha \sim_{k_a} a\beta$.
	Under the hypothesis that $A$ is eventually $I$-separable, Lemma~\ref{lem:morphismsagain}  gives that $a\alpha\gamma=a\beta\gamma$ for all $\gamma\in I^k$ where $k=k(I)$.  Under the hypothesis that $A$ is finitely generated, fix a generating set $X$ and choose $k\in\mathbb{N}$ to exceed $k_x$ for all $x \in X$. Then
	$x\alpha\,\sim_k\, x\beta$ for all $x\in X$, giving $x\alpha\gamma=x\beta\gamma$ for all $\gamma\in I^k$. Since
	$\alpha\gamma$ and $\beta\gamma$ are morphisms agreeing on a generating set, this implies that $a \alpha\gamma = a\beta\gamma$ for all $\gamma \in I^k$ and $a \in A$. In either case,  $\alpha\gamma = \beta\gamma$ for all $\gamma \in I^k$.
\end{proof}

\begin{lemma}\label{prop:rhobar} Let $I$ be a subsemigroup of $\en(A)$.
For all $(\lambda,\rho)\in  \Omega(I)$ 
	\[\overline{\rho}\colon I/{\approx}  \to  I/{\approx}, \,\, [[\alpha]]\overline{\rho}=[[\alpha\rho]]\]
	is a right translation of $I/{\approx}$.
	\end{lemma}
	\begin{proof} We first show that $\overline{\rho}$ is well defined. To this end, suppose that $\alpha,\beta\in I$ with
	$\alpha\approx \beta$. Let $a\in A$. From $a\alpha\,\sim\, a\beta$ we have that $a\alpha\gamma=a\alpha\gamma$ for all $\gamma\in I^k$, where 
	$k\geq 1$. Then for all $\gamma\in I,\delta\in I^{k-1}$ we have
	\[\begin{array}{crcl}
	& a\alpha\lambda(\gamma)\delta&=&a\beta\lambda(\gamma)\delta\\
	\Rightarrow& a(\alpha\rho)\gamma\delta&=&a(\beta\rho)\gamma\delta\\
	\Rightarrow &a(\alpha\rho)&\sim& a(\beta\rho).
	\end{array}\]
	As this is true for all $a$, we have that $\alpha\rho\approx\beta\rho$. A routine check shows that $\overline{\rho}$ is a right translation of
	$I/{\approx}$.
	\end{proof}
	
\begin{lemma}\label{lem:thereduction} Let $I$ be a subsemigroup of $\en(A)$ and let $f\in T=T(A,I)$. Define $\overline{f}\in\T_{A/{\sim}}$ by $[a]\overline{f} = [af]$ for all $a\in A$.
Then $\overline{f}$ is an $I/{\approx}$-morphism and so lies in $T(A/{\sim},I/{\approx})$, and if  $f\in S=S(A,I)$, then $\overline{f}\in S(A/{\sim},I/{\approx})$.
Consequently, $(\lambda_{\bar{f}}, \rho_{\bar{f}})\in \Omega(I/{\approx})$.
\end{lemma}
\begin{proof} Let $f\in T=T(A,I)$; we need to show that ${\bar{f}}$ is well defined. To this end, note that if 
$a,b\in A$ and $a\,\sim\, b$, then there exists $k$ such that $a\gamma=b\gamma$ for all $\gamma\in I^k$. But choosing
$\gamma=\gamma_1\dots \gamma_k$ where $\gamma_i\in I$ for $1\leq i\leq k$ we note that
\[af\gamma=af\gamma_1\dots \gamma_k=a(f\gamma_1)\dots \gamma_k=b(f\gamma_1)\dots \gamma_k=bf\gamma_1\dots \gamma_k=
bf\gamma,\]
since $f\gamma_1\in I$. Thus $af\,\sim\, bf$ so that $\bar{f}$ is well defined. It follows in a standard manner that if $f\in S=S(A,I)$, then $\overline{f}\in S(A/{\sim},I/{\approx})$.
\end{proof}

\begin{proposition}\label{prop:thereduction} Let $I$ be a subsemigroup of $\en(A)$. Suppose that $A$ is any of finitely generated, $I$-representable or eventually $I$-separable.
\begin{enumerate}\item There is a natural homomorphism 
	$\sigma : \Omega(I) \to \Omega(I / {\approx})$ given by $(\lambda,\rho)\sigma = (\bar\lambda,\bar\rho)$
	where $\bar\lambda, \bar\rho \colon I/{\approx} \to I/{\approx}$ are well-defined by 
	$$\bar\lambda([[\alpha]]) = [[\lambda(\alpha)]] \quad \text{ and } \quad [[\alpha]]\bar\rho = [[\alpha \rho]]$$
	for all $[[\alpha]] \in I/{\approx}$.
	\item If $\phi\in T(A, I)$, then $(\lambda_\phi, \rho_\phi)\sigma = (\lambda_{\bar{\phi}}, \rho_{\bar{\phi}})$.
\end{enumerate}
\end{proposition}
\begin{proof}
(1) From Lemma~\ref{prop:rhobar} we know that $\overline{\rho}$ is a well-defined right translation of $I/{\approx}$. We turn our attention to showing that
$\overline{\lambda}$ is well defined. To this end, let  $\alpha, \beta\in I$ with  $\alpha\approx \beta$.  

Suppose first that $A$ is finitely generated or eventually $I$-separable. By 
Lemma~\ref{lem:nicedescription}, there exists $k\in\N$ such that
	$\alpha\gamma = \beta\gamma$ for all $\gamma \in I^k$.   Since $\lambda$ is a left translation of $I$, we have
	$$(\lambda(\alpha))\gamma = \lambda(\alpha\gamma) =\lambda(\beta\gamma) = (\lambda(\beta))\gamma$$
	for all $\gamma \in I^k$, which means by Lemma~\ref{lem:nicedescription} again that $\lambda (\alpha) \approx \lambda (\beta)$, as required to
	show that $\bar\lambda$ is well-defined.
	
	Now consider the other hypothesis, that $A$ is $I$-representable. Let $a\in A$. We may write 
	$a=t(x_1\alpha_1,\dots, x_n\alpha_n)$ for some term $t$, $x_i\in A$ and $\alpha_i\in I$, where $1\leq i\leq n$.
	Let $b=t(x_1\alpha_1\rho,\dots, x_n\alpha_n\rho)$. By assumption that $\alpha\approx \beta$ we have
	$b\alpha\,\sim\, b\beta$, so that $b\alpha\gamma=b\beta\gamma$ for all $\gamma\in I^k$, for some $k\geq 1$.
	For every $\gamma\in I^k$ we deduce:
	\[\begin{array}{rcl}
	a\lambda(\alpha)\gamma&=&t(x_1\alpha_1,\dots, x_n\alpha_n)\lambda(\alpha)\gamma\\
	&=&t(x_1\alpha_1\lambda(\alpha),\dots, x_n\alpha_n\lambda(\alpha))\gamma\\
	&=&t(x_1(\alpha_1\rho)\alpha,\dots, x_n(\alpha_n\rho)\alpha)\gamma\\
	&=&t(x_1(\alpha_1\rho),\dots, x_n(\alpha_n\rho))\alpha\gamma= b\alpha\gamma= b\beta\gamma = \cdots=a\lambda(\beta)\gamma.
	\end{array}\]
	Thus $a\lambda(\alpha)\,\sim\, a\lambda(\beta)$ and since $a$ was arbitrary it follows that $\lambda(\alpha)\approx\lambda(\beta)$.
	Hence $\overline{\lambda}$ is well defined.

As for $\overline{\rho}$, it is easy to see that $\overline{\lambda}$ is a left translation and that 
$(\overline{\lambda},\overline{\rho})$ form a bi-translation.

	Thus, the map $\sigma : \Omega(I) \to \Omega(I/{\approx})$ is well-defined.
	
	Finally, if the images of the bi-translations $(\lambda, \rho)$, $(\lambda',\rho')$ and $(\lambda\lambda', \rho\rho')$ in $\Omega(I)$ are respectively $(\bar{\lambda}, \bar{\rho})$, $(\bar{\lambda'},\bar{\rho'})$ and $(\overline{\lambda\lambda'}, \overline{\rho\rho'})$ in $\Omega(I/{\approx})$, then for all $[[\alpha]]\in I/{\approx}$ we have that 
	$$ \overline{\lambda\lambda'}([[\alpha]]) = [[(\lambda\lambda')(\alpha)]] = [[\lambda(\lambda'(\alpha))]] = \bar{\lambda}([[\lambda'(\alpha)]]) = \bar{\lambda}(\bar{\lambda'}([[\alpha]])),$$
	so that $\overline{\lambda\lambda'} = \bar{\lambda}\bar{\lambda'}$. Dually, 
	 $\overline{\rho\rho'} = \bar{\rho}\bar{\rho'}$.
	Therefore $\sigma$ is a morphism from $\Omega(I)$ to $\Omega(I/{\approx})$.

(2) From Lemma~\ref{lem:thereduction} we know that $(\lambda_{\bar{\phi}}, \rho_{\bar{\phi}})\in\Omega(I/{\approx})$. It remains to show that 
$(\lambda_{\bar{\phi}}, \rho_{\bar{\phi}}) =(\lambda_\phi,\rho_\phi)\sigma$. To see this, let $a\in A$, $\alpha\in I$. Then
\[ [a] \overline{\lambda_\phi}( [[\alpha]]) = [a] [[\lambda_\phi (\alpha)]] = [a] [[\phi\alpha]]
				 = [a \phi \alpha] = [a\phi] [[\alpha]] = [a] \overline{\phi} [[\alpha]] 
				= [a] \lambda_{\overline{\phi}} ([[\alpha]]), \]
showing that $\overline{\lambda_\phi} = \lambda_{\overline{\phi}}$. A dual argument for $\rho_\phi$ gives us that $\overline{\rho_\phi} = \rho_{\overline{\phi}}$ and thus $(\lambda_\phi,\rho_\phi)\sigma = (\overline{\lambda_\phi},\overline{\rho_\phi})=(\lambda_{\overline{\phi}},\rho_{\overline{\phi}})$.	
\end{proof}

In view of Corollary~\ref{cor:hom} and Proposition~\ref{prop:thereduction} the following questions are natural.

\begin{question} In the case where
	$A$ is finitely generated or eventually $I$-separable, can $\Omega(I)$ be described as an extension  of the subsemigroup $\Omega(I/{\approx})$
	of $\en(A/{\sim})$?
\end{question}

\begin{definition}\label{defn:lift} Let $A$ be an algebra and let $I$ be a subsemigroup of $\en(A)$. We say that
$\hat{f}\in\fullTA$ is a {\em lift} of $f\in T(A/{\sim}, I/{\approx})$ if $[a\hat{f}]=[a]f$ for all $a\in A$.
\end{definition}

\begin{question}
	Is the idealiser of $I/{\approx}$ exactly the set of endomorphisms of $A/{\sim}$ which lift to endomorphisms of $A$?
\end{question}

We end this section by considering how lifting maps as above would work.

\begin{lemma}\label{lem:lifting}  Let $A$ be an algebra and $I$ be a subsemigroup of $\en(A)$ such that $A$ is eventually $I$-separable, and let $k=k(I)$. Let
	 $f \in T(A/{\sim}, I/{\approx})$ and let $\hat{f}: A \rightarrow A$ be a lift of $f$. 
	\begin{enumerate}
		\item For each $\alpha \in I^k$ there is a well-defined map $f\star \alpha : A \rightarrow A$ given by $f \star \alpha := \hat{f}\alpha$, which is independent of the choice of lift $\hat{f}$. Furthermore, if $f$ is an endomorphism, then so is $f\star\alpha$.
		\item If  $\hat{f}$ can be chosen to lie in $T(A,I)$,  then  $f \star \alpha  \in I^{k}$ (and hence $f \star \alpha \in I$). 
	\end{enumerate} 
\end{lemma}

\begin{proof}
	(1) Let  $\hat{f}: A \rightarrow A$ be a lift of $f$. For any $\alpha \in I^k$, we clearly have that $\hat{f}\alpha$ is a well-defined map $A\rightarrow A$. Suppose then that $\hat{f}$ and $\bar{f}$ are lifts of $f$, so  $a\hat{f} \sim a \bar{f}$  for all $a \in A$. Since $\alpha \in I^k$, this gives $a\hat{f}\alpha = a \bar{f} \alpha$ for all $a \in A$, that is $\hat{f}\alpha = \bar{f}\alpha$. Thus the definition of $f \star \alpha$ does not depend upon the choice of lift. 
	In the case where $f$ is an endomorphism, for all terms $t(a_1,\dots,a_n)$ with $a_i\in A$ we have that:
	\begin{align*}
	[t(a_1\hat{f},\dots,a_n\hat{f})] = t([a_1\hat{f}],\dots,[a_n\hat{f}])
	&= t([a_1]f,\dots, [a_n]f) &\text{since $\hat{f}$ is a lift of $f$}\\
	&= t([a_1],\dots,[a_n])f &\text{as }f\in\en(A/{\sim})\\
	&= [t(a_1,\dots,a_n)]f &\text{since $\sim$ is a congruence} \\
	&= [t(a_1,\dots,a_n)\hat{f}] &\text{by definition of $\hat{f}$}.
	\end{align*}
	So for all $\alpha\in I^k$, we have $t(a_1\hat{f},\dots,a_n\hat{f})\alpha = t(a_1,\dots,a_n)\hat{f}\alpha$, and hence (since $\alpha$ is an endomorphism) we have $t(a_1\hat{f}\alpha, \dots, a_n\hat{f}\alpha) = t(a_1,\dots,a_n)\hat{f}\alpha$. Therefore  $\hat{f}\alpha\in\en(A)$.
	
	(2) If $\hat{f}\in T(A,I)$ and $\alpha=\alpha_1\cdots \alpha_k\in I^k$, then $\hat{f}\alpha=
	(\hat{f}\alpha_1)\alpha_2\cdots \alpha_k\in I^k$.
\end{proof}

As a converse to (1) in Lemma~\ref{lem:lifting} we have the following. 

\begin{lemma}\label{lem:leftimorphims}  
Let $A$ be an algebra and $I$ be a subsemigroup of $\en(A)$ such that $A$ is eventually $I$-separable. Let $k=k(I)$, and let $\bar{f} \in \fullTA$.  Suppose further that $A$ is $I$-representable, that $(\lambda,\rho_{\bar{f}})\in\Omega(I)$ and that $\bar{f}\alpha$ is a morphism for all $\alpha\in I$. Then $\bar{f}$ is a lift of a morphism $f $ of $A/{\sim}$.
\end{lemma}
\begin{proof} We show first that $a \sim b$ implies that $a\bar{f} \sim b\bar{f}$. To this end, write
$a=t(x_1\alpha_1,\dots, x_n\alpha_n)$ and consider that for all $\alpha\in I^k$
\[\begin{array}{rcl}
a\lambda(\alpha)&=&t(x_1\alpha_1,\dots, x_n\alpha_n)\lambda(\alpha)\\
&=&t(x_1\alpha_1\lambda(\alpha),\dots, x_n\alpha_n\lambda(\alpha))\\
&=&t(x_1(\alpha_1\rho_{\bar{f}})\alpha,\dots, x_n(\alpha_n\rho_{\bar{f}})\alpha)\\
&=&t(x_1(\alpha_1\bar{f})\alpha,\dots, x_n(\alpha_n\bar{f})\alpha)\\
&=&t(x_1\alpha_1(\bar{f}\alpha),\dots, x_n\alpha_n(\bar{f}\alpha))\\
&=&t(x_1\alpha_1,\dots, x_n\alpha_n)\bar{f}\alpha\\
&=&a\bar{f}\alpha.\end{array}\]
It follows that if $a \sim b$ then 
$a\bar{f}\sim b\bar{f}$. Now defining $f$ by $[a]f=[a\bar{f}]$, we know that 
$f$ is a morphism and $\bar{f}$ is by definition a lift of $f$.
\end{proof}

Suppose now that $I$ is an ideal and $A$ is $I$-representable but not necessarily $I$-separable. By Lemma~\ref{lem: right_trans} and Theorem~\ref{cor:rep} we know that 
$$\chiR: \en(A) \hookrightarrow  \Rhot(I) \cong \Omega(I).$$
In particular, every right translation of $I$ is linked to at most one left translation of $I$. 

Let $(\lambda, \rho) \in \Omega(I)$. The main work in the previous section was to construct a \emph{well-defined endomorphism} $f: A \rightarrow A$ such that   $t(a_1, \ldots, a_n)\alpha f=t(a_1, \ldots, a_n)(\alpha\rho)$ for all $a_i \in A$, all $\alpha \in I$ and all terms $t$. The existence of such an endomorphism allowed us to conclude that $\rho=\rho_f$. The $I$-representability and $I$-separability conditions together allowed us to construct $f$ by setting $t(x_1\alpha_1, \ldots, x_k\alpha_k)f:=t(x_1(\alpha_1\rho), \ldots, x_k(\alpha_k\rho))$. We are now in a position to show that if we have both these conditions we can construct an endomorphism of $A/{\sim}$. 

\begin{proposition}\label{prop:lift}
	Let $A$ be an algebra and $I$ an ideal of $\en(A)$ such that $A$ is $I$-representable, and let  $\rho \in \Rhot(I)$.
	\begin{enumerate}
		\item There exists a well-defined morphism $f_\rho: A/{\sim} \rightarrow A/{\sim}$ defined by $$[t(x_1\alpha_1, \ldots, x_k\alpha_k)]f_\rho:=[t(x_1(\alpha_1\rho), \ldots, x_k(\alpha_k\rho))].$$
		\item In the case where $A$ is eventually $I$-separable with $k(I)=1$, the unique left translation paired with $\rho$ is defined by $\lambda(\alpha) = f_\rho\star \alpha$.
	\end{enumerate} 
\end{proposition}
\begin{proof} (1) Since $A$ is $I$-representable, from Proposition~\ref{prop:thereduction} we have
that $\bar{\rho}\in \Rhot(I/{\approx})$ where as usual we regard $I/{\approx}$ as a subsemigroup of $\en(A/{\sim})$. By Theorem \ref{prop:fixedrepandsep}, the quotient algebra $A/{\sim}$ is both $I/{\approx}\,$-representable and $I/{\approx}\,$-separable, and so by Theorem~\ref{thm:main} there is a morphism $f_{\rho} \in \en(A/{\sim})$ such that
$\bar{\rho}=\rho_{f_{\rho}}$. Moreover, from the proof of Proposition~\ref{prop:rtranslationsbymaps}, we know that $f_{\rho}$ must be given by
\[t([x_1][[\alpha_1]],\dots, [x_n][[\alpha_n]])f_{\rho}=t([x_1][[\alpha_1]]\bar{\rho},\dots, [x_n][[\alpha_n]\bar{\rho}]).\]
The claim now follows.

	(2)  By $I$-representability we know that there is a unique left translation $\lambda$ paired with $\rho$. We show that for all $\beta \in I$ we have $\lambda(\beta) = f_\rho \star \beta$.
	
	First notice that for all $b=s(y_1\alpha_1, \ldots, y_p \alpha_p) \in A$ and all $\beta \in I$ we have
	\begin{eqnarray*}
		b\lambda(\beta)	 = s(y_1 \alpha_1, \ldots, y_p \alpha_p)\lambda(\beta) &=& s(y_1 \alpha_1\lambda(\beta), \ldots, y_p \alpha_p\lambda(\beta))\\
		&=& s(y_1 (\alpha_1\rho)\beta, \ldots, y_p (\alpha_p\rho)\beta)
		= s(y_1 \alpha_1\rho, \ldots, y_p \alpha_p\rho)\beta.
	\end{eqnarray*}

	Now let $\hat{f}: A \rightarrow A$ be any lift of $f_\rho$, that is, $[a\hat{f}] = [a]f_\rho$ for all $a \in A$. Then, by the definition of $f_\rho$ and the fact that $\hat{f}$ is a lift of $f_\rho$ we have:
	$$[s(y_1 \alpha_1\rho, \ldots, y_p \alpha_p\rho)] = [s(y_1 \alpha_1, \ldots, y_p \alpha_p)]f_\rho = [b]f_\rho = [b\hat{f}].$$
	Thus for all $\beta \in I$ we have that 
	$$b\lambda(\beta) = s(y_1 \alpha_1\rho, \ldots, y_p \alpha_p\rho)\beta =  b\hat{f}\beta,$$
	where the first equality comes from the previous paragraph, and the second equality follows from the fact that $\beta \in I$ cannot distinguish elements in the same $\sim$-class.
\end{proof}

\begin{remark}
Suppose that $A$ is $I$-representable and eventually $I$-separable with $k(I)=1$. Then by the previous lemma we may define a function from $\Rhot(I)$ to $\en(A/{\sim})$ by $\rho \mapsto f_\rho$ and it is straightforward to check that this is a semigroup morphism. Moreover, by part (2) of Proposition~\ref{prop:lift} we have that $f_{\rho_1} = f_{\rho_2}$ if and only if $\rho_1$ and $\rho_2$ are linked to the same left translation $\lambda$. In the case where $A$ is actually $I$-separable, by definition of $\sim$ we have that $A = A/{\sim}$, and it follows from the results of the previous section that each left translation is paired with at most one right translation. In this case, $f_\rho$ is an endomorphism of $A$ and the morphism $\rho \mapsto f_\rho$ is just the inverse of the morphism from $\en(A)$ to $\Rhot(I)$ given by $f \mapsto \rho_f$.
\end{remark}

\section{Examples and applications}\label{sec:free algebras}

We begin by presenting three examples of  general situations and behaviours that arise later in our specific examples. These are essentially folklore. For Example~\ref{lright zero} we refer the reader to \cite{Tamura:1958}.
\begin{example}
	\label{full}
	If $I = \en(A)$, then it is clear that $\Omega(I) \cong \en(A)$. Indeed, for every monoid $M$ with identity element $e$ we have $\Omega(M) \cong M$. For, if  $\rho$ is a right translation of $M$, then in particular, $a\rho =a(e\rho)$ for all $a \in M$. Thus each right translation of $M$ is a translation by right multiplication by some element $f=e\rho \in M$; call this map $\rho_f$. Likewise, each left translation of $M$ is a translation by left multiplication by some element $g=\lambda(e) \in M$; call this map $\lambda_g$. Moreover the pair $(\lambda_g, \rho_f)$ is linked if and only if $f=g$ since the requirement that $e\lambda_ge = e\rho_f e$ gives $f=g$.\end{example}

\begin{example}
	\label{null}
	Suppose that $I$ is a null semigroup, that is, there exists $\gamma \in I$ such that $\alpha\beta = \gamma$ for all $\alpha, \beta \in I$. Then it is straightforward to see that \emph{any} left/right map fixing $\gamma$ is a left/right translation, and moreover that \emph{any two such maps are linked}. Indeed, a right map $\rho$ is a right translation if and only if for all $\alpha,\beta \in I$ we have $\alpha(\beta\rho) = (\alpha\beta)\rho$, which since $I$ is null is equivalent to $\gamma=\gamma\rho$. A dual argument holds for left translations. Moreover, it is clear that if $\lambda$ and $\rho$ are left/right maps fixing $\gamma$, then for all $\alpha,\beta \in I$ we have $\alpha(\lambda(\beta)) =\gamma = (\alpha \rho) \beta$. Thus the translational hull of $I$ is
$$\Omega(I) = \{(\lambda, \rho): \lambda, \rho: I \rightarrow I, \lambda(\gamma)=\gamma=\gamma \rho\}.$$
\end{example}

	\begin{example}
		\label{lright zero} Suppose that $I$ is a right zero semigroup, that is, $\alpha\beta=\beta$ for all $\alpha\in I$. In this case, 
		for every map $f:I \rightarrow I$ we have $\rho_f$ is a right translation, and the only left translation is the identity map $\iota$ on $I$.
		Then $\Omega(I)$ consists of all pairs $(\iota,\rho_f)$ where $f\in \fullTA[I]$ and hence $\Omega(I)$ is isomorphic to $\fullTA[I]$.
		The result for left zero semigroups is dual.
\end{example}

\subsection{Free Algebras}\label{sub:free}

We now consider the special case where $A$ is freely generated by a set $X$ (that is, every map $X\rightarrow A$ extends (necessarily uniquely) to some 
$\alpha\in\en(A)$) and where  $I$ is one of the following ideals of $\en(A)$:
$$I_k:=\{ \alpha\in\en(A): \im\alpha\subseteq \langle Y\rangle, \mbox{ for some }Y \subseteq A, |Y| < k\},$$
for some cardinal $k$. The ideal $I_1$ is the set of endomorphisms whose images contain only constants of the algebra. If $k>|X|$, then it is clear that $I_k = \en(A)$. Thus we will often assume that $k\leq |X|$. Notice that $A$ is $I_k$-representable if and only if $k\geq 2$, or $k=1$ and $A=C$, where $C=\langle \emptyset\rangle$ is the subalgebra generated by the images of the basic nullary operations; recall that our convention is that $C=\emptyset $ if $A$ has no such operations. On the other hand, we may not necessarily have $I_k$-separability, as this happens only when $I_k$ is right reductive, as given by the following lemma. 

\begin{lemma}\label{lem:SEP in free alg} Let $A$ be a free algebra and 
suppose that $k \geq 2$. The following are equivalent:
\begin{enumerate}
	\item $A$ is $I_k$-separable;
	\item $I_k$ is right $\fullTA$-reductive;
	\item $I_k$ is right $\en(A)$-reductive;
	\item $I_k$ is right reductive.
\end{enumerate}
\end{lemma}
\begin{proof}
Suppose that (1) holds. Then by part (7) of Lemma~\ref{lem:reductivity from rep or sep} we have that $I_k$ is right $\fullTA$-reductive so that (2) holds. That (2) implies (3) and (3) implies (4) is immediate from the definitions.

Now assume that (4) holds. Let $a, b\in A$ with $a \neq b$, and consider the maps $\alpha,\beta\in\EndA$ defined by $x\alpha = a$ and $x\beta = b$ for all $x\in X$.
Then $\im\alpha\subseteq \clotX[a]$ and $\im\beta\subseteq \clotX[b]$, which shows that $\alpha,\beta\in I_k$ since $k \geq 2$.
Suppose for contradiction that $a\gamma = b\gamma$ for all $\gamma\in I_k$.
Then we get that $x\alpha\gamma = x\beta\gamma$ for all $x\in X$, and thus $\alpha\gamma = \beta\gamma$ for all $\gamma\in I_k$ as $A$ is freely generated by $X$.
Since $I_k$ is right reductive, it follows that $\alpha = \beta$, thus  for every $x \in X$ we have $a = x\alpha = x\beta = b$, a contradiction.
Therefore for each $a, b\in A$ with $a \neq b$, there exists $\gamma\in I_k$ with $a\gamma\neq b\gamma$, that is, $A$ is $I_k$-separable.
\end{proof}

Lemma \ref{lem:SEP in free alg} demonstrates that for the most natural ideals $I_k$ of $\en(A)$, where $A$ is a free algebra, the conditions of $I_k$-separability and right reductivity are the same. For an algebra $A$ and ideal $I$ of $\en(A)$ more generally, this need not be the case as the following example illustrates. 

\begin{example}\label{ex:rightreductivenosep} We use a standard construction from the theory of semigroups; for further details see, for example, \cite{Howie:1995}. Let $A$ be a semigroup that has decomposition as  a semilattice $Y$ of non-trivial groups $G_{\alpha}$, $\alpha\in Y$. This means that $Y$ is a commutative semigroup of idempotents, such that $A=\bigcup_{y\in Y}G_{y}$, each $G_y$ is a non-trivial subgroup of $A$, and $G_y G_z\subseteq G_{yz}$. We identify $y\in Y$ with the identity $e_{y}$ of $G_{y}$.   Let 
	\[I=\{\alpha\in\en(A): \im\alpha\subseteq Y\}.\]
	Clearly $I$  is a left ideal of $\en(A)$. Since $Y$ is characteristic in $A$, that is, $Y\alpha\subseteq Y$ for all $\alpha\in\en(A)$, we also have that $I$ is a right ideal, hence an ideal, of $\en(A)$.
	To see that $I$ is right reductive, consider the map $\delta\in I$ defined by $g_y\delta=y$ for all $y\in Y$, $g_y\in G_y$. 
	Then, if $\alpha,\beta\in I$ and $\alpha\gamma=\beta\gamma$ for all $\gamma\in I$, we get that $\alpha\delta=\beta\delta $.  Since $\delta$ is the identity on the image of any element of $I$, we obtain $\alpha=\beta$, so that $I$ is right reductive. 
	However, $A$ is not $I$-separable. Indeed, let $y\in Y$ and let $g_y\in G_y$ with $g_y\neq y$. 
	Since any morphism of $A$ preserves subgroups, and the subgroups of $\im\gamma$ for each $\gamma\in I$ are trivial, we must have that $g_y\gamma=y\gamma$ for each $\gamma\in I$. Since $I$ is right reductive it follows from \cite{Petrich:1970} that $\Omega(I)$ is naturally isomorphic to $\Lambdat(I)$. \end{example}

\begin{question}\label{qn:rightreductivenosep} In Example~\ref{ex:rightreductivenosep}, $I$ is not right $\en(A)$-reductive. 
	Does there exist an algebra $A$ and an ideal $I$ of $\en(A)$ such that $I$ is right $\en(A)$-reductive but $A$ is not $I$-separable (or even not weakly $I$-separable, in the sense defined in Remark~\ref{remark:weaksep})?\end{question}

\begin{proposition}
Let $A$ be a free algebra, and $k \geq 2$ be such that $I_k$ is right reductive. Then $\Omega(I_k)$ is canonically isomorphic to $\en(A)$.
\end{proposition}
\begin{proof}
By Lemma \ref{lem:SEP in free alg} $A$ is $I_k$-separable, so Theorem \ref{thm:main} applies to give that $\Omega(I_k) = \{(\lambda_f, \rho_f): f \in \en(A)\} \cong \en(A)$. 
\end{proof}
Our next goal is to show that for all $k \geq 2$ regardless of whether $I_k$ is right reductive we  have that every bi-translation of $I_k$ is realised by {\em transformations} (which are not necessarily morphisms) of $A$. 

\begin{lemma}\label{lem:Ik} 
\begin{enumerate}
\item For every $k\leq|X|$ we have $I_k=I_k^2$.
\item If $k \geq 2$, then every right translation of $I_k$ is right-balanced and moreover $$\Omega(I_k) = \{(\lambda_f, \rho_f): f \in T(A, I_k)\} \cong T(A, I_k)/{\equiv_{\im I_k}}.$$
\end{enumerate}\end{lemma}
\begin{proof}(1) Let $\alpha\in I_k$, so that $\im\alpha\subseteq \langle  Y\rangle$, where $|Y|<k$. 
	Let $Y=\{ y_j:j\in J\}$ where $|J|=|Y|$.
	For each $x\in X$ let $x\alpha=t_x(y_{j_1},\dots, y_{j_n})$ where $n\geq 0$ and $y_{j_\ell}\in Y$ for $1\leq \ell \leq n$. 
	Let $X=Z\cup W$ where $|Y|=|Z|$ and $Z\cap W=\emptyset$, write $Z=\{z_j:j\in J\}$ and let $y\in Y$ be fixed. Now define $\beta,\gamma\in \en(A)$ as follows:
	\[x\beta=t_x(z_{j_1},\ldots, z_{j_n})\,\,\mbox{ for }x\in X \]
	and
	\[z_j\gamma=y_j, \,\, j\in J\mbox{ and } w\gamma=y \mbox{ for all } w \in W.\]
	 It is easy to see that $\beta,\gamma\in I_k$ and $\alpha=\beta\gamma$. 

(2) This follows immediately from Proposition \ref{prop:suffmaps}, noting that if $k \geq 2$ then for each $c \in A$ there exists a morphism in $I_k$ with $x \gamma_c = c$ for all $x \in X$, and so in particular $A$ is $I_k$-representable.
\end{proof}

The proof given in Proposition \ref{prop:suffmaps} can be modified slightly to show that in the case of a free algebra $A$ the translational hull of many more subsemigroups of $\en(A)$ will be realised by mappings.
\begin{theorem}\label{thm:free} Let $A$ be  a free algebra and let $I$ be any subsemigroup of $\en(A)$ such that
$I$ contains $I_2$. Then
\[\Omega(I)=\{ (\lambda_f,\rho_f): f\in T(A,I)\}.\]\end{theorem}
\begin{proof} Let $X$ be a set of free generators for $A$.  First note that, as in the previous proof, $A$ is $I$-representable since if  $c\in A$ we may define $\gamma_c$ by $x\gamma_c=c$ for all $x\in X$, and clearly
$\gamma_c\in I_2 \subseteq I$. Thus by Lemma \ref{lem:lbalanced} we have that every linked left translation is left-balanced.

We show that any $\rho\in \Rho(I)$ is right-balanced. Let $a,b\in A$ and $\alpha,\beta\in I$ with $a\alpha=b\beta$. 
Clearly $\gamma_a\alpha=\gamma_b\beta$. 
It follows that \[\gamma_a(\alpha\rho)=(\gamma_a\alpha)\rho=(\gamma_b\beta)\rho=\gamma_b(\beta\rho).\]
Thus 
\[a(\alpha\rho)=x\gamma_a(\alpha\rho)=x\gamma_b(\beta\rho)=b(\beta\rho)\]
where $x$ is taken to be any element of $X$, hence showing that $\rho$ is right-balanced. The result then follows from Theorem \ref{thm:maps}.
\end{proof}

We can say a little more about the semigroups $T(A,I_k)$. \begin{proposition}\label{prop:kmorph} Let $A$ be a free algebra. Then $f\in T(A,I_k)$ if and only if
	\begin{enumerate}\item 
		$f|_{\lowerme{\langle Y\rangle}}:\langle Y\rangle\rightarrow A$ is a morphism for all $Y \subseteq  A$ with $|Y| < k$;
		and
		\item for all $t(x_1,\ldots, x_n)\in A$ we have $t(x_1,\ldots, x_n)f\,\sim \, t(x_1f,\ldots, x_nf)$.
	\end{enumerate}
\end{proposition}
\begin{proof} Since $I_k=I_k^2$, we may apply parts (3) and (6) of Lemma \ref{lem: right_trans} to find $f \in T(A, I_k)$ if and only if for all $\alpha \in I_k$ we have: $f|_{\im\alpha}$ is a morphism and $t(x_1f, \ldots, x_nf) \alpha = t(x_1, \ldots, x_n) f\alpha$ for all $t(x_1, \ldots, x_n) \alpha \in A$. Noting that for every $Y \subseteq  A$ with $|Y| < k$ there exists $\alpha\in I_k$ with $\im\alpha=\langle Y\rangle$, the result follows.  
\end{proof} 

\begin{example}
\label{ex:nonneg}
Consider the free monogenic semigroup $A = (\mathbb{N}, +)$. The endomorphisms of $A$ are the maps $\varphi_a: \mathbb{N} \rightarrow \mathbb{N}$ defined by $n \varphi_a = na$ for all $n \in \mathbb{N}$.  It follows that for all $a,b\in \mathbb{N}$ we have $\varphi_a\varphi_b=\varphi_{ab}$. In this case $I_1 = \emptyset$ and $I_2 = \en(A)$, and so the results of the current section do not give much of the overall flavour. 

However, it is easy to see that $\en(A)$ is isomorphic to $(\mathbb{N},\times)$ and so commutative and cancellative. Let $I = \{\varphi_{2k}: k \in \mathbb{N}\}$. The preceding comments give that   $I$ is an ideal of $\en(A)$ so that $S=S(A,I)=\en(A)$. Further, $I$ is  left and right $\en(A)$-reductive. Moreover, since every element of $\en(A)$ is injective, it is immediate that $A$ is $I$-separable. However, $A$ is not $I$-representable, since the submonoid of $A$ generated by the images of elements in $I$ is $2 \mathbb{N} \neq \mathbb{N}$. Thus in particular, the converse to part (1) of Lemma \ref{lem:reductivity from rep or sep} does not hold.
	
Now suppose that $\rho$ is a right translation of $I$ with $\varphi_{2} \rho = \varphi_{2k}$ for some $k \in \mathbb{N}$. Since multiplication is commutative, for all $m \in \mathbb{N}$ we see that:
	$$\varphi_{4mk} = \varphi_{2m}\varphi_{2k} = \varphi_{2m}(\varphi_{2})\rho =(\varphi_{4m})\rho = \varphi_{2}(\varphi_{2m})\rho \Rightarrow (\varphi_{2m})\rho = \varphi_{2mk}.$$
	That is, $\rho = \rho_{\varphi_{k}}$. This demonstrates that $\Rho(I) = \Rho_S(I)$, and so $\RhoS(I)=\Rhot(I) = \Rho(I)$. It is then easy to see that in this case $\chiR$ is an isomorphism from $\en(A)$ to $\Rhot(I)=\Rho(I)$.
	
	A similar calculation reveals that every left translation $\lambda$ of $I$ is of the form $\varphi_h$ and $\chiL$ is an isomorphism from $\en(A)$ to $\Lambdat(I)=\Lambda(I)$. Finally, cancellativity in $\en(A)$ gives that if $(\lambda_{\varphi_h},\rho_{\varphi_k})\in \Omega(I)$, then $h=k$.  Thus in this case, $\en(A)$ is naturally isomorphic to $\Omega(I)$.
\end{example}

\begin{example}

Let $A$ be the free monoid on two elements $\{ a,b\}$. The elements of $A$ are
words (that is, finite strings) of elements of $\{ a,b\}$; the operation is concatenation and the 
identity is the empty word $\varepsilon$. We consider the ideal 
\[I_2=\{ \alpha\in\en(A): \im\alpha\subseteq \langle u\rangle, u\in A\},\]
where for $u\in A$ we  have
$\langle u\rangle=\{ u^k:k\in\mathbb{N}_0\}$
is the submonoid of $A$ generated by $u$.
Given that $A$ is free, any morphism $\alpha$ is determined by the values it takes on $a$ and $b$;
if $\alpha\in I_2$ it takes $a$ and $b$ to powers of the same word. If $a\alpha=u^m$ and $b\alpha=u^n$   we write
$\alpha=u(m,n)$. 
We may assume in the above that $u$ is chosen to be {\em primitive}, that is,
$u\neq v^k$ for any word $v\neq u$. 
For $x\in \{ a,b\}$ and $u\in A$ we  denote by $|u|_x$ the number of occurrences of letter $x$ in $u$. 
Notice that for all $v,w\in A$ and $u(m,n)\in I_2$ we have 
\[\begin{array}{rcll}
v\sim w&\Leftrightarrow &v\sim_1 w&\\
&\Leftrightarrow & vu(m,n)=wu(m,n)&\mbox{ for all }u(m,n)\in I_2\\
&\Leftrightarrow & |v|_am+|v|_bn=|w|_am+|w|_bn&\mbox{ for all }m,n\in \mathbb{N}_0\\
&\Leftrightarrow &|v|_a=|w|_a\mbox{ and }|v|_b=|w|_b.
\end{array}
\]
It follows that $A$ is not $I_2$-separable.
 
 It is easy to see that
$f: A \rightarrow A$ is in the right idealiser of $I_2$ in $\fullTA$ if and only if for all words 
$u\in A$ and $k\in \mathbb{N}_0$ we have $(u^k)f=(uf)^k$.
On the other hand $f: A\rightarrow  A$ is in the left  idealiser of $I_2$ in $\fullTA$ if and only
if for all $u\in A$ we have
\[|uf|_a= |u|_a|af|_a + |u|_b|bf|_a\mbox{ and }
|uf|_b= |u|_a|af|_b + |u|_b|bf|_b,  \]
that is, $f$ is linearly increasing the number of letters $a$ and $b$ in the word, but does not
recognise the order in which they appear.  In the notation of Section~\ref{sec:nosep}, if $f$ is in the idealiser $T(A,I_2)$, we have
$\rho_f\mapsto \overline{\rho_f}=\rho_{f'}\in \Rho(I/{\approx})$ where $f'=f_{\rho_f}$. A lift of
$f'$ is a map $f'':A\rightarrow A$ such that for all $w\in A$ we have
$[wf'']=[w]f'=[wf]$. Since $f\alpha\in I_2$ for all $\alpha\in I_2$,  $I_2=I^2_2$ and $\sim$ is a congruence, we obtain that for every $a_1\dots a_n\in A$, where $a_i\in \{ a,b\}$, 
$1\leq i\leq n$, we have
$a_1f''\dots a_nf''\sim (a_1\dots a_n)f''$. This is equivalent to $f''$ being in the left  idealiser of $I_2$ in $\fullTA$.

A concrete example of $f$ in $T(A,I_2)$ that is {\em not} in $\en(A)$,
is given by the map $(a_1\dots a_n)f= a_n\dots a_1$.
\end{example}

\subsection{Independence algebras}\label{sub:ind}
Let $A$ be an independence algebra, that is, a free algebra on a set $X$ which also satisfies the exchange property \cite{Gould-IndepAlg}.
In particular, this means that any maximal independent set $X$ (that is, a set $X$ such that
for all $x\in X$ we have $x\notin\langle X\setminus\{ x\}\rangle$) is a minimal generating set for $A$, called a {\em basis} of $A$, and that any independent set can be extended to a basis of $A$.
Moreover, every subalgebra admits a basis, and all bases of a subalgebra $B\subseteq A$ have the same cardinality, called the {\em dimension} of $B$.
Hence, for each map $\alpha\in\EndA$, we denote by $\rk\alpha$ the \emph{rank of $\alpha$}, corresponding to the dimension of the subalgebra $\im\alpha$.
Independence algebras were formerly called $v^*$-algebras \cite{narkiewicz} and were utilised by the first author in~\cite{Gould-IndepAlg}. Notable examples
include sets, free group acts and vector spaces. 

The proper ideals of the endomorphism monoid $\EndA$ are well-known and form a chain, corresponding exactly to the sets $I_k$ of endomorphisms with rank less than the cardinal $1\leq k\leq |X|$, that is, $I_k := \set{\alpha\in\EndA : \rk\alpha < k }$.
In particular, these ideals are the same as those considered in the previous subsection and are now the {\em only} ideals of $\en(A)$.
In the case where $A$ is a vector space, the ideal $I_2$ is completely 0-simple. Translational hulls of completely 0-simple semigroups $S$ have been well studied \cite{Petrich:1968,Petrich:1970}, including in the case where $S=I_2$ for $A$ a vector space \cite{Gluskin:1959a,Petrich:1970} and in this case it is known that the translational hull is isomorphic to $\en(A)$. Returning to the general case, it is clear from definition that $A$ is $I_k$-representable for all $k \geq 2$; we give the following characterisations of $I_k$-separability in this case.

\begin{theorem}\label{cor:equivalences indep alg}
Let $A$ be an independence algebra  with basis $X$ and let $|X| \geq k\geq 2$. Then the following are equivalent:
	\begin{enumerate}
		\item $A$ is $I_k$-separable;
		\item $\Omega(I_k)\isom\Rhot(I_k) \isom \Lambdat(I_k) \isom \EndA$;
		\item $I_k$ is a reductive ideal;
		\item $k>2$, or $k=2$ and every $1$-dimensional subalgebra contains at least two elements.
	\end{enumerate}
\end{theorem}
\begin{proof}
	Recall that $A$ is always $I_k$-representable when $k\geq 2$.
	Under this observation, we have that $(1)\Rightarrow (2)$ by Corollary~\ref{thm:main} and $(2)\Rightarrow (3)$ by Theorem \ref{lem:reductivity from rep or sep}.
	The fact that $(3)\Rightarrow (1)$ comes directly from Lemma~\ref{lem:SEP in free alg} since a reductive ideal is right reductive.
	We show the  equivalence of $(1)$ and $(4)$.
Suppose first that $A$ is $I_k$-separable. If $k>2$, there is nothing to show, so suppose that $k = 2$.  Let $a\neq b\in A$. 
Since $A$ is $I_2$-separable, there exists $\gamma\in I_2$, such that $a\gamma \neq b\gamma$, and hence $\im\gamma$ is a subalgebra of $A$ containing at least two elements. Since $\gamma \in I_2$, by definition we have that $\im\gamma$ is contained in a $1$-dimensional subalgebra of $A$. Because all $1$-dimensional subalgebras are isomorphic in an independence algebra, it follows that every one-dimensional subalgebra contains at least two elements, as required.
	
To conclude the proof we need to show that if $k>2$ or $k=2$ and every $1$-dimensional subalgebra of $A$ contains at least two elements, then we must have that $A$ is $I_k$-separable. To this end, let $a\neq b\in A$ and as above let $C=\langle \emptyset\rangle$.

\begin{itemize}
	\item If $a,b\in C$, then $a\gamma = a \neq b= b\gamma$ for all $\gamma\in I_k$, so assume from now on that $a\notin C$ so that $\{a\}$ is an independent set. There are then two possibilities: either $b=t(a)$ for some term $t$, or $\set{a,b}$ is independent.	
	\item If $b=t(a)$, we extend the independent set $\{a\}$ to a basis $Y$ of $A$. Taking $\gamma$ to be the unique endomorphism mapping all basis elements in $Y$ to $a$, we have $\im\gamma \subseteq \clotX[a]$, so that $\gamma\in I_2 \subseteq I_k$ and $b\gamma = t(a)\gamma = t(a\gamma) = t(a) = b\neq a = a\gamma$. 
	\item Suppose then that $\set{a,b}$ is independent and extend this to a basis $Y$ of $A$. If $k\geq 3$, taking $\gamma$ to be the unique endomorphism mapping $b$ to $b$ and all remaining basis elements in $Y$ to $a$, we have $\im\gamma \subseteq \clotX[a,b]$, so that $\gamma\in I_3 \subseteq I_k$ and $a\gamma = a \neq b = b\gamma$. On the other hand, if $k=2$ and all $1$-dimensional subalgebras of $A$ are not singletons, it follows that there exists $s \in \langle a \rangle$ with $s \neq a$, and in this case taking $\gamma$ to be the unique endomorphism mapping $b$ to $s$ and all remaining basis elements in $Y$ to $a$ we have $\im\gamma \subseteq \clotX[a]$, so that $\gamma\in I_2 \subseteq I_k$ and $a\gamma = a \neq s = b\gamma$.
\end{itemize}

Hence, in all cases, there exists a map $\gamma\in I_k$ such that $a\gamma \neq b\gamma$, which shows that $A$ is $I_k$-separable.
\end{proof}

When $k = 1$ the situation is slightly more complicated. Note that $I_1$ exists (recalling that our ideals are taken to be non-empty) if and only if  $C=\langle\emptyset\rangle\neq\emptyset$, in which case $I_1$ is easily seen to be the set of endomorphisms with image equal to $C$. Thus, as we remarked earlier, $A$ is $I_1$-representable if and only if $A=C \neq \emptyset$. The following lemma characterises when $A$ is $I_1$-separable, and demonstrates (as we shall see in Remark \ref{rem:indnorepsep} below) that we need not have equivalence between statements $(1)$ and $(3)$ of Theorem~\ref{cor:equivalences indep alg} in the case $k=1$.

\begin{lemma}\label{lem:SEP for I1 in indep alg}
	Let $A$ be an independence algebra,  $C=\langle\emptyset\rangle$ and suppose that $C \neq \emptyset$ so that $I_1$ is an ideal of $\en(A)$. Then $A$ is $I_1$-separable if and only if $|C|\geq 2$ and for all unary terms $t$, if $t(c) = c$ for all $c \in C$ then $t(a) = a$ for all $a \in A$.
\end{lemma}
\begin{proof}
	Notice first that $\im\gamma = C$ for all $\gamma\in I_1$.
	
	If $A$ is $I_1$-separable then for $a\neq b\in A$, there exists $\gamma\in I_1$ such that $a\gamma \neq b\gamma$.
	Thus $\Card{C}\geq 2$. For the remaining part of this implication we prove the contrapositive. If $t$ is a unary term such that $t(a) \neq a$ for some $a \in A$, then for all independent elements $a\in A$, we have that $t(a)\neq a$.
	Then, by $I_1$-separability, there exists $\gamma\in I_1$ such that $a\gamma\neq t(a)\gamma = t(a\gamma)$, that is, there exists $c\in C$ such that $c\neq t(c)$.
	
	Conversely, suppose that $\Card{C}\geq 2$ and that for each unary $t$ such that there exists $a \in A$ with $t(a) \neq a$, we have an element $c_t\in C$ such that $c_t \neq t(c_t)$.
	Let $a\neq b\in A$.
	If $a,b\in C$, then clearly $a\gamma\neq b\gamma$ for all $\gamma\in I_1$, so suppose without loss of generality that $a\notin C$, so that $\set{a}$ is independent. We then have one of the following situations:
	\begin{itemize}
		\item $b\in C$; in this case since $\Card{C}\geq 2$ we may choose $d \in C$ with $d \neq b$; since $\lbrace a \rbrace$ is an independent set there is a morphism $\gamma \in I_1$ mapping $a$ to $d$, and hence $b \gamma = b \neq d = a\gamma$.
		\item $b = t(a)$ for some unary term $t$; in this case because $\lbrace a \rbrace$ is an independent set there is a morphism $\gamma \in I_1$
		mapping $a$ to $c_t$; and $b = t(a)$ to $t(a)\gamma= t(a \gamma)= t(c_t) \neq c_t$;
		\item $\set{a,b}$ is independent; in this case we may choose any $c, d \in C$ with $c \neq d$ and there will be a $\gamma \in I_1$ mapping $a$ to $c$ and $b$ to $d$.
	\end{itemize}
	In all cases, we get that $\gamma\in I_1$ with $a\gamma\neq b\gamma$ and thus $A$ is $I_1$-separable.
\end{proof}

\begin{remark}\label{rem:indnorepsep}
Suppose that $C \neq \emptyset$, so that $I_1  = \{\alpha \in \en(A): \im \alpha = C\} \neq \emptyset$. Since the elements of $C$ are fixed points of every endomorphism of $A$, we have that $I_1$ is a left-zero semigroup ($\gamma\alpha = \gamma\beta = \gamma$ for all $\alpha,\beta,\gamma\in I_1$). Thus $I_1$ is in particular always right reductive, but $A$ is $I_1$-separable only under the conditions given in Lemma~\ref{lem:SEP for I1 in indep alg}. (For example, if $A$ is a vector space then $I_1$ is the ideal consisting of the zero map, in which case $A$ is obviously not $I_1$-separable.) To see that the condition that $|C|\geq 2$ is not sufficient for the converse statement in Lemma~\ref{lem:SEP for I1 in indep alg} to hold, consider a non-trivial group $G$ as a free left $G$-act, and let
$\{ p,q\}$ be a 2-element $G$-act where the action of $G$ is trivial. Then $G\cup\{ p,q\}$, equipped with the unary $G$-act operations $\{ t_g:g\in G\}$ and nullary operations
$\nu_p$ and $\nu_q$ with images $p$ and $q$ respectively,  is an independence algebra with $|C|=2$. If $g$ is a non-identity element of $G$,
then $t_g(h)=gh\neq h$ for any $h\in G$, but $t_g(c)=c$ for $c\in C=\{ p,q\}$.  

Similarly, if $C \neq A$, then $A$ is not $I_1$-representable, and $I_1$ is left reductive if and only if $C$ is a singleton.
Indeed, since $I_1$ is a left-zero semigroup, we either have that $I_1$ is not left reductive, or $\alpha = \beta$ for all $\alpha,\beta\in I_1$, that is, $I_1$ is a singleton. In the latter case, since $C\neq A$, there exists an independent element $a\in A$, and thus for all $c\in C$, there exists a map $\gamma\in I_1$ which sends $a$ to $c$. The fact that $I_1$ is a singleton forces $C$ to be a singleton, as required.
\end{remark}
Theorem~\ref{cor:equivalences indep alg} gives that in many cases the translational hull of an ideal of an independence algebra $A$ is naturally isomorphic to the endomorphism  monoid $\en(A)$. The following result completes the description of the translational hull of an ideal of an independence algebra.

\begin{theorem}
Let $A$ be an independence algebra  and let $I$ be a (non-empty) ideal of $A$. Then the translational hull $\Omega(I)$ is isomorphic to $\en(A)$ or  $\fullTA[I]$. If $\Omega(I)$ is isomorphic to $\fullTA[I]$, then $I$ is the minimal ideal.
\end{theorem}

\begin{proof}
Recall that each proper ideal of $A$ is equal to $I_k$ for some $1 \leq k \leq |X|$ where $X$ is a basis for $A$. In the case where $k>2$ or $k=2$ and each $1$-dim subalgebra contains at least two elements, it follows immediately from Theorem~\ref{cor:equivalences indep alg} that the translational hull of $I_k$ is naturally isomorphic to the endomorphism  monoid $\en(A)$.

If $k=2$ and each $1$-dimensional subalgebra is a singleton, then $I_2$ is then the set of all endomorphisms whose image is a singleton set, which in turn forces the subalgebra $C$  to be empty (since $\Card{X}\geq k=2$ and for all $x\in X$, we have that $C \subseteq \clotX[x]$).
It is then  clear that $I_2 = \{\gamma_u: u \in A\}$ where for each $u \in A$, $\gamma_u$ denotes the constant map defined by $a\gamma_u = u$ for all $a \in A$ (which is an endomorphism by our assumption that every $1$-dimensional subalgebra is a singleton). In this case $I_2$ is the minimal ideal of $\en(A)$, and is easily seen to be a right zero semigroup, so Example~\ref{lright zero} applies to give that $\Omega(I_2)$ is isomorphic to $\fullTA[I_2]$.

Finally, if  $k=1$ then $I_1$ is non-empty if and only if $C$ is non-empty. Supposing then that $C$ is non-empty we then have that $I_1 = \{\gamma: \im\gamma = C\}$ is the minimal ideal, and it is a left zero semigroup, so by Example~\ref{lright zero} the translational hull of $I_1$ is isomorphic to $\fullTA[I_1]$.
\end{proof}

\begin{remark}
Notice that the two possibilities given in the previous result need not be mutually exclusive e.g.\ if the independence algebra $A$ is a set, then it is easy to see that the translational hull $\Omega(I_2)$ of the minimal ideal $I_2$ is isomorphic to \emph{both} $\en(A)=\fullTA[n]$ and $\fullTA[I_2]$.
Moreover, the converse of the second statement need not hold e.g.\ if the independence algebra $A$ is an affine algebra with distinguished non-zero subspace $A_0$, the minimal ideal is $I_2$ (since $C=\emptyset$), and (since the $1$-dimensional subalgebras are not singletons) the translational hull $\Omega(I_2)$ is isomorphic to $\en(A)$, which is \emph{not} isomorphic to $\fullTA[I_2]$ in general. 
\end{remark}
\subsection{Matrix monoids and endomorphisms of free modules over semirings}
A \textit{semiring} is an algebra defined on a set $R$ with two associative binary operations $+$ and $\times$ such that $\times$ distributes across $+$ from both sides. We say that $R$ is commutative if $\times$ is commutative, we say that $R$ has an identity if there is a constant $1$ of the algebra with the property that $1 \times r = r= r \times 1$ for all $r \in R$, and we say that $R$ has a zero if there is a constant $0$ of the algebra with the properties that $0 + r = r$ and $0 \times r = 0 = r \times 0$  for all $r \in R$. Fields and (commutative) rings are examples of (commutative) semirings; other well-studied examples include the \textit{tropical} and \textit{boolean}
semirings. The main applications of our next results are to $\en(A)$ where  where $A$ is a free semimodule over a semiring. However, we present a more general abstract approach, motivated by classical ring and semiring theory.

\begin{proposition}
\label{semiring}

Let $R$ be a (not necessarily commutative) semiring with $1$ and 
suppose $I$ is a semigroup ideal of the multiplicative semigroup of $R$ which is not contained in any proper semiring ideal of $R$. Then the translational hull of $I$ is naturally isomorphic to the multiplicative semigroup $R$, that is, $\chiO^R$ is an isomorphism.
\end{proposition}

\begin{proof}
By the assumption on the ideal we can write $1$ as a finite sum of 
elements in $I$, say $1 = a_1+\cdots+a_k$. It follows that $R$ acts faithfully by multiplication on both the left and right of $I$. It follows from this that distinct elements of $R$ determine distinct bi-translations of $I$, i.e. the map $\chiO^R$ is a natural embedding.

Conversely, if $(\lambda,\rho) \in \Omega(I)$ then, using the above decomposition of $1$, for every $c$ in $I$ we have
\begin{eqnarray*}
\lambda (c) = 1 (\lambda (c)) = (a_1+\cdots+a_k) \lambda (c) &=& a_1 \lambda (c) +\cdots + a_k 
\lambda (c)\\
 &=& (a_1 \rho )c+\cdots + (a_k \rho)c = (a_1 \rho + \cdots + a_k \rho) c
\end{eqnarray*} so $\lambda$ is just left
translation by the element $f:=a_1 \rho +\cdots +a_k \rho \in R$. A dual argument shows that $\rho$ is right translation by $g:=\lambda( a_1) + \cdots  + \lambda( a_k) \in R$. Now since $(\lambda, \rho)$ is a linked pair for all $a, b \in I$ we have $a f b = a g b $, and using again the fact that $R$ acts faithfully on $I$, we obtain $f = g$ and  $(\lambda,\rho) = (\lambda_f, \rho_f)$. Thus $\chiO^R$ is surjective.
\end{proof}

\begin{remark}
\label{matrices}Let $R$ be a semiring with $1$ and $0$, and let $M_n(R)$ denote the set of all $n \times n$ matrices with entries from $R$. It is easy to see that $M_n(R)$ is a semiring with respect to the obvious induced operations of matrix addition and matrix multiplication, and has multiplicative identity element. Taking $I$ to be a semigroup ideal of the multiplicative monoid of $M_n(S)$ that is large enough to generate $M_n(R)$ as a semiring ideal (e.g. any semigroup ideal containing the $0$-$1$ matrices with exactly one entry equal to $1$), the previous result applies to show that the translational hull of $I$ is naturally isomorphic to $M_n(R)$. Of course, $M_n(R)$ may be identified with the endomorphism monoid of the free $R$-module $R^n$. \end{remark}

Our next proposition will allow us to extend the above remark  to arbitrary free $S$-modules.

\begin{proposition}
\label{semiring+} 
Let $R$ be a semiring, $I$ a semigroup ideal of the multiplicative semigroup of $R$, and $E=\{ e_i:i\in B\}$ a set of idempotents of $I$ such that:
\begin{enumerate} \item  for every finite set $F\subseteq I$ we have
\[p=p(\Sigma_{k\in K}e_k)\]
for all $p\in F$, some finite subset $K=K_F$  of $B$;

\item 
for every subset
\[C=\{ c_i:i\in B\}\subseteq I\]
there exists $q=q_C \in R$ such that $e_ic_i=e_iq$ for all $i\in B$;

\item  $e_ia=e_ib$ for all $i\in B$, where $a,b\in R$,  gives $a=b$. 
\end{enumerate}
Then the translational hull $\Omega(I)$ is naturally isomorphic to the multiplicative semigroup $R$,  that is, $\chiO^R$ is an isomorphism.
\end{proposition}

\begin{proof} 
 Let $(\lambda,\rho)\in \Omega(I)$ and let $C=\{ e_i\rho:i\in B\}$. Let  $q$ be the element guaranteed by (2) and let 
$x\in I$. By (1) we have
\begin{equation}
\label{sub}
x=x(\Sigma_{k\in K}e_k)
\end{equation}
for some finite set $K\subseteq B$. Now take the finite subset $F=\{ x\rho, e_k\rho: k\in K\} \subseteq I$, 
and set $L=K_F$ as given by (1).
Then since $x\rho \in F$ by using (1), distributivity in $R$ and the fact that $(\lambda, \rho)$ is a bi-translation we have:

$$x\rho = x\rho \Sigma_{l\in L}e_l =\Sigma_{l\in L}(x\rho)e_l =\Sigma_{l\in L}x\lambda(e_l).
$$
Substituting the expression for $x$ given in \eqref{sub} into the right-hand side of the above, using distributivity and the fact that $(\lambda, \rho)$ is a bi-translation once more then gives:
$$x\rho=\Sigma_{l\in L}(x(\Sigma_{k\in K}e_k)\lambda(e_l)=\Sigma_{l\in L}\Sigma_{k\in K}xe_k\lambda(e_l)=\Sigma_{l\in L}\Sigma_{k\in K}x(e_k\rho)e_l.$$
By using distributivity once more, we see that the right hand sum can be rearranged as follows, upon which we may utilise (1) again since each $e_k \rho \in F$:
$$x\rho=\Sigma_{k\in K}x(e_k\rho(\Sigma_{l\in L}e_l))=\Sigma_{k\in K}x(e_k\rho).$$
Now since each $e_k$ is idempotent and $\rho$ is a right translation, the right-hand side can be re-written as follows:
$$x\rho=\Sigma_{k\in K}xe_k(e_k\rho).$$
Using the fact that each $e_k \rho \in C$, by (2) we have that $e_k(e_k\rho) = e_kq$, and so by using distributivity and \eqref{sub} once more we have
$$x\rho=\Sigma_{k\in K}xe_kq=x(\Sigma_{k\in K}e_k)q=xq.$$

Notice that the element $q$ is independent of the element $x$ and so the above argument shows that $x\rho=xq$ for all $x \in I$. That is, $\rho$ is the right translation $\rho_q$. Now for every $y\in I$ we have
\[e_i\lambda(y)=(e_i\rho)y=e_iqy\]
for every $i\in B$, so that $\lambda(y)=qy$ by (3). That is, $\lambda$ is the left translation $\lambda_q$. Thus we have shown that for every bi-translation $(\lambda,\rho)$ there exists an element $q \in R$ such that $(\lambda, \rho)=(\lambda_q,\rho_q)$, and by using (3) again it is clear that $q$ is unique. Thus $\Omega(I)$ is naturally isomorphic to $R$.
\end{proof}

\begin{remark}
	\label{additive}
	Suppose that $A$ is an algebra with a binary operation $+$ such that $(A,+)$ is a commutative monoid with identity element $0_A$, then $\en(A)$ becomes a semiring with respect to the operations of composition and point-wise addition (i.e. $a(\alpha+\beta) = a\alpha +a \beta$ for all $a \in A$). The identity map is the identity of $\en(A)$ whilst the zero map ($a \mapsto 0_A$ for all $a \in A$) is a zero element for $\en(A)$, and so one may apply the previous result to this situation.
\end{remark}

\begin{remark} Let $S$ be a semiring, $F$ a free $S$-module and $R = \en(F)$. By the previous remark we may regard $R$ as a semiring. Let $\{ z_i:i\in B\}$ be a basis for  $F$ and let $e_i\in R$ be defined by
\[z_je_i=\delta_{ij}z_j.\]
Now let  $I$ be  any semigroup ideal of $R$ containing only morphisms with image contained in a finitely generated subsemimodule,  and including the subset $E=\{ e_i:i\in B\}$. Let $C=\{ c_i:i\in B\}\subseteq I$ and define $\alpha\in R$ by $z_i\alpha=z_ic_i$. Then the conditions of the proposition hold for this $E$, giving that the translational hull $\Omega(I)$ is naturally isomorphic to the endomorphism monoid $\en(F)$.
\end{remark}

\subsection{Finite symmetric groups}
\label{sec:Sn}

As a very different application of our earlier results, we conclude the paper 
 by considering in detail the case where $I$ is an ideal of the endomorphism monoid of a finite symmetric group $\S_n$ with identity element denoted by $e$. In what follows we shall describe the ideal structure of $\en(\S_n)$ in the case where $n>2$ and $n$ is not $4$ or $6$ so that the kernel of any endomorphism of $\S_n$ must be exactly one of $\{e\}$, $\mathcal{A}_n$ or $\S_n$, and, all automorphisms are inner. We leave the determination of the ideal structure of $\en(\S_n)$ in each of the remaining cases as an exercise for the interested reader; this can be computed in a very similar fashion, upon noting the following differences: case $n=1$ is trivial; for $n=2$ the alternating group coincides with the trivial group, so there are only two possible kernels of an endomorphism; for $n=4$ there is another normal (Klein $4$) subgroup, and hence another possible kernel; and for $n=6$ there are outer automorphisms.
 
 Clearly the only endomorphism with kernel $\S_n$ is $\phi_e\colon \S_n\rightarrow \{e\}$, and this is a zero of $\en(\S_n)$. Since all automorphisms are inner, the endomorphisms with trivial kernel are precisely the maps $\psi_s\colon \sigma \rightarrow s^{-1}\sigma s$ where $s \in \S_n$. The endomorphisms with kernel $\mathcal{A}_n$ are precisely the maps $\phi_t$ indexed by elements $t$ of order $2$ in $\S_n$  and defined by $s \phi_t = e$ if $s \in \mathcal{A}_n$  and $s \phi_t = t$ if $s \in \mathcal{O}_n:=\S_n\setminus\mathcal{A}_n$. It shall be convenient to write $D = \{t \in \S_n: t^2=e\}$.
With notation as above it is straightforward to check that:
\begin{enumerate}
	\item $\psi_{s}\psi_t = \psi_{st}$ for all $s,t \in \S_n$;
	\item $\phi_t\psi_s = \phi_{s^{-1}ts}$ for all $s \in \S_n$, $t \in D$;
	\item $\psi_s\phi_t= \phi_t$ for all $s\in \S_n$, $t \in D$;
	\item If $s \in D$ is even: $\phi_s\phi_t = \phi_e$ for all $t \in D$;
	\item If $s \in D$ is odd: $\phi_s\phi_t = \phi_t$ for all $t \in D$.
\end{enumerate}
It is convenient here to use the notion of left and right divisibility of elements, encoded by Green's relations $\mathcal{R}$ and $\mathcal{L}$, and their associated pre-orders; we refer the reader to a standard text such as \cite{Howie:1995}  for the details.
Clearly the automorphisms $\psi_s$ form the group of units of $\en(\S_n)$ and this is isomorphic to $\S_n$.  It follows from (2) that if $t, r \in D$ are elements of the same cycle structure then $\phi_t$ and $\phi_r$ are $\mathcal{R}$-related.  Notice that for $s \in D$ even, it follows from (4) and the fact that $\phi_e$ is a zero of the monoid that no other elements are $\mathcal{R}$-related to $\phi_s$.  On the other hand, if $s \in D$ is odd then it follows from (5) that $\phi_s$ is $\mathcal{R}$-related to every $\phi_t$ where $t \in D$ is odd. Moreover, each $\phi_s$ with $s \in D$ odd is $\mathcal{R}$-above every $\phi_t$ with $t \in D$ even. Finally, notice that by (3), (4) and (5), regardless of whether $t \in D$ is odd or even the $\mathcal{L}$-class of each $\phi_t$ is a singleton.

The egg-box diagram of this endomorphism monoid is therefore as shown in Figure~\ref{fig1}. Writing $K$ to denote the set of all $\phi_p$ with $p \in D$ and $p$ even, and $E$ to denote the set of all $\phi_p$ with $p \in D$ with $p$ odd, notice that $\phi_e$ is a zero element, $K$ is a null semigroup, $E$ is right-zero semigroup, and each element of $E$ acts as a left identity on $\en(\S_n) \setminus {\rm Aut}(\S_n) = E \cup K$.
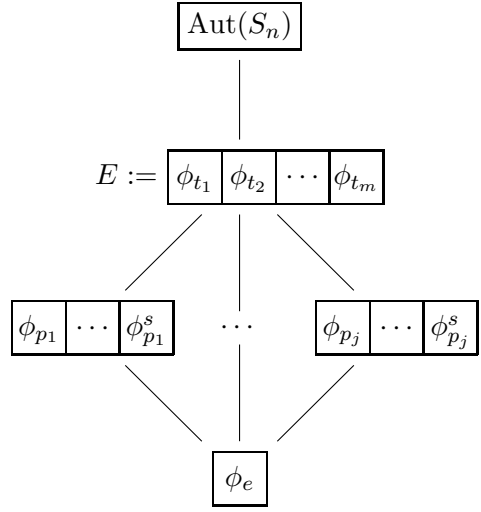
\begin{figure}
	\label{fig1}
\begin{center}
\begin{tikzpicture}
\node (aut) at (0,4) {\framebox{${\rm Aut}(S_n)$}};
\node (a) at (0,2) {$E:=$\ytableausetup{centertableaux, boxsize=1.8em}
	\begin{ytableau}
	\phi_{t_1} & \phi_{t_2} & \cdots & \phi_{t_m}\\
	\end{ytableau}};
\node (d) at (-2,0) {\ytableausetup{centertableaux}
	\begin{ytableau}
	\phi_{p_1} & \cdots & \phi_{p_1}^s\\
	\end{ytableau}};
\node (e) at (0,0) {$\cdots$};
\node (f) at (2,0) {\ytableausetup{centertableaux}
	\begin{ytableau}
	\phi_{p_j} & \cdots & \phi_{p_j}^s\\
	\end{ytableau}};
\node (zero) at (0,-2) {\begin{ytableau}
	\phi_e\\
	\end{ytableau}};
\draw (zero) -- (d) -- (a) -- (aut);
\draw (zero) -- (e) -- (a);
\draw (zero) -- (f) -- (a);
\end{tikzpicture}
\end{center}
\caption{The egg-box diagram of $\en(\S_n)$. The automorphism group ${\rm Aut}(S_n)=\{\psi_s: s \in \S_n\}\cong \S_n$ is the $\mathcal{H}$-class of the identity element (i.e.\ the group of units);  $E:=\{\phi_{t_1}, \ldots, \phi_{t_m}\}$ is a single $\mathcal{R}$-class containing all (idempotent) endomorphisms $\phi_{t}$ with $t \in S_n$ odd and of order $2$; the set of all $\phi_p$ with $p \in S_n$  even and of order $2$ decomposes as a disjoint union of $\mathcal{R}$-classes (one for each conjugacy class of even elements of order $2$) of the form $\{\phi_p^s: p \in S_n \mbox{ even of order } 2, s \in \S_n\}$, where we write simply $\phi_p^s$ to denote $\phi_{s^{-1}ps}$, and take $p_1, \ldots, p_j$ to be a transversal of conjugacy classes of even order $2$ elements; and  $\{\phi_e\}$ is the minimal ideal. }
\end{figure}

Consider now an ideal $I$ of $\en(\S_n)$. This must be a union of $\mathcal{R}$-classes that is downwards closed with respect to the $\mathcal{R}$-order. We now consider all possible ideals $I$ of $\en(\S_n)$ and calculate the translational hulls.\\

\noindent
\textbf{The case \boldmath{$I=\en(\S_n)$}}. As noted above (see Example \ref{full}), it is clear that  $$\Omega(I) = \big\{(\lambda_f, \rho_f): f \in \en(\S_n)\big\} \cong \en(\S_n).$$ However, as we shall see below, there is no proper ideal $I$ with $\Omega(I) \cong \en(\S_n)$.

\medskip
\noindent
\textbf{The case \boldmath{$I=\en(\S_n) \setminus {\rm Aut}(\S_n)$}}. Noting that $I = \{ \phi_t: t \in D \}$ it is clear from the definition of the endomorphisms $\phi_t$ that the pair $\S_n$ is $I$-representable, but is not $I$-separable, since if permutations $a$ and $b$ have the same sign, then they cannot be separated by any element of $I$. It is also easy to see that $I^2=I$ in this case. Now, noting that $\S_n$ has a generating set $X$ consisting of all transpositions, and for each $t \in D = \bigcup_{\gamma \in I} \im\gamma$ there exists an endomorphism $\phi_t \in I$ with $x \phi_t = t$ for all $x \in X$, Proposition \ref{prop:suffmaps} immediately gives that $\Omega(I) = \{(\lambda_f, \rho_f):f \in T(\S_n, I)\}$.

To determine the elements of $T(\S_n, I)$, observe that if $f\in \mathcal{T}_{\mathcal{S}_n}$, then $fI\subseteq I$ if and only if
either $f:\mathcal{A}_n\rightarrow \mathcal{A}_n$ and $f:\mathcal{O}_n\rightarrow \mathcal{O}_n$, or $f:\mathcal{S}_n\rightarrow \mathcal{A}_n$; let us refer to this property of $f$ as being {\em strongly parity preserving}. We leave the reader to check
these are precisely the maps in $\mathcal{T}_{\mathcal{S}_n}$ we would obtain via the process given in Proposition~\ref{prop:lift} (2). On the other hand, $If\subseteq I$ if and only if 
$ef=e$ and $Df\subseteq D$. Thus 
\[T(\S_n, I)=\{ f\in  \mathcal{T}_{\mathcal{S}_n}: ef=e, Df\subseteq D\mbox{ and }f\mbox{ is strongly parity preserving}\}.\] 
Returning to Proposition \ref{prop:suffmaps}, we have that $\Omega(I)\cong T(\S_n, I)/{\equiv_{\im}}$. We have observed  $\bigcup_{\gamma\in I}\im\gamma=D$
and so for $f,g\in T(\S_n, I)$ we deduce $f\equiv_{\im} g$ if and only if $Df=Dg$. 

\medskip
\noindent\textbf{The remaining cases}. Suppose now that $I$ is any ideal generated by some collection of $\phi_t$ where each $t \in D$ is even. In this case $I$ is a null semigroup, and $\S_n$ is neither $I$-representable nor $I$-separable (the image of each map in $I$ consists of even permutations, and so $\bigcup_{\alpha\in I} \im\alpha \subseteq A_n$, and just as in the previous case, elements of the ideal cannot separate elements of $\S_n$ with the same sign). Since $I$ is null, as noted in Example \ref{null} we have that the translational hull of $I$ is
$$\Omega(I) = \big\{(\lambda, \rho): \lambda, \rho: I \rightarrow I, \lambda\phi_e=\phi_e=\phi_e \rho\big\}.$$
Fixing any right translation $\rho$  it is straightforward to verify (in a similar manner to the previous case) that $\rho = \rho_f$ for any function $f: \S_n \rightarrow \S_n$ with the property that $\phi_s \rho = \phi_t$ implies $sf=t$. On the other hand, not all left translations are realised by mappings $f: \S_n \rightarrow \S_n$. Indeed, let $s, t, e$ be distinct elements of $\S_n$ and consider any map $\lambda: I \rightarrow I$ satisfying  $\lambda(\phi_s) = \phi_t$  and $\lambda(\phi_e) = \phi_e$. As in Example \ref{null},  $\lambda$ is a left translation. However, for any $f: \S_n \rightarrow \S_n$ it is easy to see that $\lambda \neq \lambda_f$. Indeed, the image of $\lambda_f(\phi_s) = f\phi_s$ is contained in $\{e, s\}$, whilst the image of $\lambda(\phi_s) = \phi_t$ is precisely $\{e, t\}$. Thus in this case, not all left translations are realised by mappings of $\S_n$.

\begin{remark}
Another example where we can encounter left and right translations which are linked but are not coming from mappings on the underlying algebra is that of $\en(\fullTA[n])$.  The ideal structure of $\en(\fullTA[n])$ may be found in \cite{Gould:2023,Grau:2023} and discussions of the translational hulls of these ideals  in \cite{Grau:2023}.
\end{remark}

\section*{Acknowledgements} The authors are grateful to Stuart Margolis, who, by asking us  disparate and well motivated questions, prompted us to make this investigation. Visits to York made by the third author were supported by both the University of Manchester and the University of York. The second author was supported by a PhD studentship from the University of York; a number of our results also appear as joint work in Chapters V-VII of his thesis \cite{Grau:2023}.

\end{document}